\documentclass[11pt,letterpaper,reqno]{amsart}

\usepackage{amsthm, amssymb, amscd, amsmath}
\usepackage{graphicx, color}

\usepackage{hyperref}
\usepackage{enumerate}

\usepackage[T1]{fontenc}
\usepackage{lmodern}

\usepackage[text={5.4in, 8.0in},centering]{geometry}				
\usepackage{microtype}

\usepackage{parskip}

\numberwithin{equation}{section}

\newtheorem{theorem}{Theorem}[section]  

\newtheorem{proposition}[theorem]{Proposition}

\newtheorem{corollary}[theorem]{Corollary}
\newtheorem{lemma}[theorem]{Lemma}

\newtheorem{assumption}{Assumption}

\newtheorem{thmx}{Theorem}

\theoremstyle{definition}
\newtheorem*{remark}{Remark}

\newcommand{\Z}{\mathbb{Z}}
\newcommand{\N}{\mathbb{N}}
\newcommand{\R}{\mathbb{R}}

\newcommand{\T}{\mathbb{T}}

\newcommand{\cD}{\mathcal{D}}
\newcommand{\cE}{\mathcal{E}}

\newcommand{\cA}{\mathcal{A}}

\newcommand{\cG}{\mathcal{G}}

\newcommand{\cL}{\mathcal{L}}

\newcommand{\cP}{\mathcal{P}}

\newcommand{\bOne}{\mathbf{1}}

\newcommand{\bmat}[1]{\begin{bmatrix}#1\end{bmatrix}}

\DeclareMathOperator*{\argmin}{arg\,min}

\newcommand{\dist}{\mathrm{dist}}
\newcommand{\st}{\, : \,}
\newcommand{\loc}{\mathrm{loc}}
\newcommand{\Cper}{C_{\mathrm{per}}}
\newcommand{\bary}{\bar{y}}
\newcommand{\tlK}{\widetilde{K}}
\newcommand{\tL}{\widetilde{\mathcal{L}}}
\newcommand{\tlh}{\tilde{h}}
\newcommand{\tlX}{\widetilde{X}}

\providecommand{\norm}[1]{\lVert#1 \rVert}
\newcommand{\normbetastar}[1]{\norm{#1}_{\beta V,*}}
\newcommand{\normstar}[1]{\norm{#1}_{*}}

\newcommand{\diag}{\mathrm{diag}\,}

\author{Konstantin Khanin}
\address{University of Toronto, Toronto, Canada}
\email{khanin@math.toronto.edu}

\author{Ke Zhang}
\address{University of Toronto, Toronto, Canada}
\email{kzhang@math.toronto.edu}

\author{Lei Zhang}
\address{Dalian University of Technology, Dalian, China}
\email{lzhang@dlut.edu.cn}

\title{Uniform exponential contraction for viscous Hamilton-Jacobi equations}
\begin{document}
\maketitle

\begin{abstract}
The well known phenomenon of exponential contraction for solutions to the
viscous Hamilton-Jacobi equation in the space-periodic setting is based on
the Markov mechanism. However, the corresponding Lyapunov exponent $\lambda(\nu)$ characterizing the exponential rate of contraction depends on the viscosity $\nu$. The Markov mechanism provides only a lower bound for $\lambda(\nu)$ which vanishes in the limit $\nu \to 0$. At the same time, in the inviscid case $\nu=0$ one also has exponential contraction based on a completely different dynamical mechanism. This mechanism is based on hyperbolicity of action-minimizing orbits for the related Lagrangian variational problem.

In this paper we consider the discrete time case (kicked forcing), and establish a uniform lower bound for $\lambda(\nu)$ which is valid for all $\nu\geq 0$. In the  proof which is based on a nontrivial interplay between the dynamical and Markov mechanisms for exponential contraction
we combine PDE methods with the ideas from the Weak KAM theory.
\end{abstract}

\section{Introduction}

We consider the periodic Hamilton-Jacobi equation with viscosity $\nu >0$:
\begin{equation}
	\label{eq:HJ-vis}
	\varphi_t + \frac12 |\nabla \varphi|^2 = \nu \Delta \varphi + F(x, t), x \in \R^d, t \in \R, 
\end{equation}
where $F(\cdot, t)$ is a $\Z^d$-periodic function, namely $F(x + k, t) = F(x, t)$ for all $k \in \Z^d$. Let $\Cper(\R^d)$ and $\Cper^s(\R^d)$ denote the space
of periodic continuous and $C^s$ functions, respectively. Given an initial condition $\varphi_0 \in \Cper(\R^d)$ and $T_0 < T$, consider the solution $\varphi_{T_0}^\nu(x, t)$ to the initial value problem 
\begin{equation}
	\label{eq:HJ-vis-ivp}
	\begin{aligned}
	& \varphi_t + \frac12 |\nabla \varphi|^2 = \nu \Delta \varphi + F(x, t),  & x \in \R^d, \, t \in [T_0, T], \\
	& \varphi(x, T_0) = \varphi_0(x), & x \in \R^d. 
	\end{aligned}
\end{equation}

Sinai (\cite{Sinai1991}) proved that if $F(x, t) = F(x) B(t)$, where $B(t)$ is either white noise or periodic, there is a unique stationary solution $\psi^\nu \in \Cper(\R^d \times \R)$, such that 
\begin{equation}
	\label{eq:limit-vis}
	\lim_{T_0 \to -\infty} \|\varphi^\nu_{T_0}(x, T) - \psi^\nu(\cdot, T)\|_* \to 0, 
\end{equation}
where $\|f\|_* = \inf_{C \in \R} \sup_x |f(x) + C|$ is the supremum norm modulo an additive constant. Moreover the convergence is exponential, namely, there is $\lambda_\nu >0$ such that 
\begin{equation}
	\label{eq:exp-vis}
	\limsup_{T_0 \to -\infty} \frac{1}{|T_0|} \log \|\varphi^\nu_{T_0}(x, T) - \psi^\nu(\cdot, T)\|_* < - \lambda_\nu. 
\end{equation}
The uniqueness of stationary solution also holds in more general settings (\cite{GIK+2005}).

In this paper, we are interested in the uniform lower bound for the exponent $\lambda_\nu$. This is related to the property of the viscosity limit $\nu \to
0$. As $\nu \to 0$, any limit point of $\varphi^\nu$ in the $\|\cdot\|_*$ norm solves the invicid equation
\begin{equation}
	\label{eq:invicid}
	\begin{aligned}
	& \varphi_t + \frac12 |\nabla \varphi|^2 = F(x, t),  & x \in \R^d, \, t \in (T_0, T), \\
	& \varphi(x, T_0) = \varphi_0(x), & x \in \R^d. 
	\end{aligned}
\end{equation}
Under certain non-degeneracy conditions (\cite{IS2009}, \cite{WKM+2000}, \cite{IK2003}), the solution to the invicid problem $\varphi^0_{T_0}(x, t)$ also admits a unique stationary solution $\psi^0$, in the same sense as before: 
\[
	\lim_{T_0 \to -\infty} \|\varphi^0_{T_0}(x, T) - \psi^0(\cdot, T)\|_* \to 0. 
\]
The exponential convergence of the invicid solution also hold under similar conditions, (see \cite{WKM+2000}, \cite{IKZ2019},) namely, there exists $\lambda > 0$ such that 
\[
	\limsup_{T_0 \to -\infty} \frac{1}{T} \log \|\varphi^0_{T_0}(x, T) - \psi^0(\cdot, T)\|_*  < -\lambda. 
\]

The exponential convergence \eqref{eq:exp-vis} comes from the diffusion term $\nu \Delta \varphi$, and \emph{a priori} we can think that $\lambda_\nu \to 0$ as $\nu \to 0$. However, the exponential convergence in the inviscid case makes it plausible to expect a uniform bound $\lambda_\nu > \bar{\lambda} > 0$ for all $\nu > 0$. The main result of this paper is the proof of this uniform bound. We should point out that the mechanism of the exponential convergence in the inviscid case is purely dynamical. Correspondingly, the main difficulty in proving the uniform exponent is to study interaction between the dynamical and Markovian mechanisms asymptotically as $n \to 0$. 

\subsection*{Statement of the main result}

We are considering the so-called kicked setting in this paper, namely, 
\[
	F(x, t) = \sum_{j \in \Z} F_j(x) \delta(t - j), 
\]
where $F_j \in \Cper^3(\R^d)$. The kicked setting retains much of the feature of the system, but is more convenient to study. We fully expect that the results remain true in the continuous time case. 

To make our arguments more transparent, we consider the simplest case, where all the kicks are the same ($F_j = F$), and $F$ is a generic potential on $\T^d$. Again, we expect the results to hold in the non-stationary case. In particular, in the case when $\{F_j\}$ form an i.i.d. sequence in the space $C^3(\T^d)$. 
\begin{assumption}\label{as:F}
\[
	F(x, t) = \sum_{j \in \Z} F(x) \delta(t - j), \quad  x \in \R^d.
\]
Here $F \in \Cper^3(\R^d)$, $\argmin_{x \in \T^d} F(x) = \{0\}$, $F(0) = 0$, and $D^2F(0)$ is strictly positive definite. 
\end{assumption}

\begin{thmx}\label{thm:unif-exp}
Suppose $F$ satisfies  Assumption~\ref{as:F}, then there exists $\lambda > 0$, $\nu_0 >0$ and $C > 0$ depending only on $F$, such that for all $\nu \in (0, \nu_0)$ and initial conditions $\varphi_0 \in \Cper(\R^d)$, the solution $\varphi^\nu_{-n}(x, t)$ to \eqref{eq:HJ-vis} satisfies
\[
	\|\varphi^\nu_{-n}(\cdot, 0) - \psi^\nu(\cdot)\|_* \le B e^{- \lambda n}, \quad n \ge 0, 
\]
where the constant $B > 0$ depends only on $\|\varphi\|_*$ and $C$. 
\end{thmx}

\subsection*{Discussions of the result and the method}
The standard way to study equation \eqref{eq:HJ-vis-ivp} is to apply the Hopf-Cole transformation $u = e^{-\varphi/(2\nu)}$, which transforms our equation to the inhomogeneous heat equation 
\begin{equation*}
	\begin{aligned}
	& u_t  = \nu \Delta u - \frac{1}{2\nu} F(x, t) u, & (x, t) \in \R^d \times (-n, 0)\\
	& u(x, -n) = u_0(x) =  e^{- \frac{1}{2\nu} \varphi_0(x)} & x \in \R^d.  
	\end{aligned}
\end{equation*}
The solution to this equation (under Assumption~\ref{as:F}) is given by $\cL_\nu^n u_0$, where the operator is $\cL_\nu = e^{\nu \Delta} e^{F/(2\nu)}$. In this sense, our result can be understood as proving a uniform spectral gap for the operator $\cL_\nu$. Our proof, however, does not use spectral theory. Instead, we convert the operator $\cL_\nu$ to a sequence of Markov operators, and prove that the inhomogeneous Markov chain converges exponentially with a uniform rate. As such, our method can be adapted to deal with the case that $F_j$'s are different while sharing some uniform properties, where the spectral approach seems to struggle. One of our main motivations is to apply this approach to the random setting of \cite{KZ2017} and \cite{IKZ2019}, which we will address in a separate paper. 

By the Feynman-Kac formula, the solution $\cL_\nu^n u_0$ can be interpreted as the expectation of an integral over a Brownian path (which becomes a random walk in the kicked case). As $\nu \to 0$, the trajectory of the random walk converges in distribution to the minimizing path of the associated Lagrangian. This is where the connection with the invicid equation \eqref{eq:invicid} lies, as the solution of the invicid equation is given precisely by the integral of the Lagrangian over a minimal path, via the Lax-Oleinik variational principle. We use Weak KAM theory to study the properties of the minimal path, which is used to obtain estimates on the operator $\cL^n_\nu$, via the classical Laplace's method. 

There is an important technical hurdle to this plan: the Laplace's method depend strongly on the dimension of the space, while the integral associated to $\cL_\nu^n u_0$ is in $\R^{nd}$. As such, the estimates fail for large $n$. In order to deal with this problem, we devise the following strategy:
\begin{enumerate}
 \item[Step 1.] We first conjugate the operator $\cL_\nu$ to an operator $\tL_\nu = e^{\psi/(2\nu)} \cL_\nu e^{-\psi/(2\nu)}$, where $\psi$ is the stationary solution of the \emph{invicid} equation. 
 \item[Step 2.] We then obtain uniform estimates for the \emph{partition function} $\tL^n \bOne$ ($\bOne$ is the constant function $1$),  for $0 \le n \le N_0(\nu)$, where $N_0(\nu) \sim (\nu \log \frac{1}{\nu})^{-\frac13}$. 
 \item[Step 3.] We show that if the estimates of $\tL^n_\nu \bOne$ in Step 2 holds for $0 \le n \le N(\nu)$, then $\tL^n$ contracts exponentially in the same time interval, using a suitable norm. 

 \item[Step 4.] We then bootstrap our estimates: if $\tL^n$ contracts exponentially, we can apply it to $\tL^N_\nu \bOne$ to get good estimates for $\tL^{n + N} \bOne$, extending item (2) to longer time intervals. 
\end{enumerate}

Heuristically, there are two mechanisms of exponential convergence. The first mechanism is that the invicid problem has a unique global minimizer supported at the minimum of $F$, and all minimizers are attracted to it at an exponential rate. The second comes from the ellipticity of the viscous equation, which provides convergence of the Markov chain, but a priori only at a rate of $O(\nu)$. The first mechanism says all minimizers are uniformly close, which is why we can estimate $\tL_\nu^n \bOne$ for large $n$. This is roughly Step 2. We then apply the second mechanism. Thanks to the good estimates in Step 2, the rate of convergence is uniform in $\nu$. This is roughly Step 3. It turns out that once the second mechanism kicks in, it is self-perpetuating via the bootstrap argument.

\subsection*{Plan of the paper}

The plan of this paper is as follows.
\begin{itemize}
 \item  In Section~\ref{sec:var} and \ref{sec:hyp} we study the invicid equation. Section~\ref{sec:var} recalls basic Weak KAM theory and the underlying Lagrangian dynamics, which (in the kicked case) is given by a twist map. Some standard proofs are provided for the benefits of the reader. 
 \item The Lagrangian system has a unique global minimizer which is a hyperbolic fixed point of the twist map. In Section~\ref{sec:hyp} we combine hyperbolic theory and variational theory to obtain estimates of the Lagrangian action.  
 \item In Section~\ref{sec:hopf-cole}, we introduce the Hopf-Cole transformation and state Theorem \ref{thm:exp-conv-heat}, which is the counter part of our main theorem after the transformation. We also state Proposition~\ref{prop:laplace} that provides the initial estimates for $\tL^n_\nu \bOne$. This is our main technical result, whose proof is postponed to the last two sections. 
 \item In Section~\ref{sec:hm}, we prove Proposition~\ref{prop:pi-hm} and Corollary~\ref{cor:contraction}, implementing Step 3 of our plan. The main tool is the Lyapunov function approach to the convergence of Markov chain (see \cite{HM2011}). 
 \item In Section~\ref{sec:boostrap}, we implement Step 4, namely the bootstrap argument. The main theorem is proved assuming Proposition~\ref{prop:laplace} holds. 
 \item Finally, Proposition~\ref{prop:laplace} is proven in Sections~\ref{sec:hess} and \ref{sec:laplace}. This uses the exponential convergence of the minimizer, and a somewhat delicate application of the Laplace's method. 
\end{itemize}

\section{The variational analysis and the weak KAM solution}
\label{sec:var}

The results of this section hold for a general $C^2$ potential $F \in \Cper(\R^d)$ such that $\min F = F(0) = 0$ and $F(x) > 0$ for all $x \notin \Z^d$. 

Let $\varphi \in \Cper(\R^d)$, define the Lax-Oleinik operator $T: \Cper(\R^d) \to \Cper(\R^d)$ by 
\[
	T(\varphi)(x) = \min_{y \in \R}\{\varphi(y) + h(y, x)\},
\]
where 
\[
	h(y, x) = \frac12 |x - y|^2 + F(y)
\]
is called the \emph{generating function}. The solution to \eqref{eq:invicid} is given by $\varphi(x, n) = T^n(\varphi_0)(x)$ for $n \in \N$. The fixed points of $T$ are the stationary solutions of \eqref{eq:invicid}. 

At this point it is convenient to consider functions and operators defined on $\T^d = \R^d/\Z^d$. Define, for $x, y \in \R^d$, 
\[
	|x|_\T = \min_{k \in \Z^d} |x + k|, \quad A(y, x) = \min_{k \in \Z^d} h(y + k, x). 
\]
Note in particular, as functions on $\R^d \times \R^d$, 
\[
	A(y, x) = h(y, x), \quad \text{ if } |y - x| < \frac12. 
\]
The same functions make sense for  $x, y \in \T^d$, which we denote by the same names. The distance on $\T^d$ is given by $|y - x|_\T$. If $\varphi \in \Cper(\R^d)$, let $\varphi'$ be the associated function on $\T^d$, define 
\[
	T'(\varphi') = \min_{y \in \T^d}\{\varphi'(y) + A(y, x)\}
\]
then $T'(\varphi')$ lifts to $T(\varphi)$ in $\R^d$. From now on we will use the same notation for a function in $\Cper(\R^d)$ and its counter part in $C(\T^d)$, and $T$ for the operators in both spaces.

A function $\varphi: \R^d \to \R$ is called $C$-semi-concave if for each $x \in \R^d$, there exists $l \in \R^d$ such that
\[
	\varphi(y) - \varphi(x) \le l \cdot (y - x) + C |y - x|^2. 
\]
The vector $l$ is called a super-gradient, and we use $\partial \varphi(x)$ to denote the set of all super-gradients. The function $\varphi$ is differentiable at $x$ if and only if $\partial \varphi(x)$ is a singleton. A function $\varphi \in C(\T^d)$ is semi-concave if the associated function on $\Cper(\R^d)$ is semi-concave. 

The following properties of semi-concave function are useful. 
\begin{lemma}[Proposition 4.7.3 of \cite{Fathi2008}]\label{lem:semi-concave-prop}
\begin{enumerate}[(1)]
	\item  If $\varphi \in \Cper(\R^d)$ is $C$-semi-concave, then it is $C'$-Lipschitz with $C'$ depending only on $C$ and the dimension $d$. 
	\item  Suppose $-\varphi_1,  \varphi_2$ are $C$-semi-concave functions on $\T^d$, then for each  $x \in \argmin\{\varphi_2 - \varphi_1\}$,  $d\varphi_1(x)$, $d\varphi_2(x)$ exists and are equal.
	In particular, any periodic semi-concave function $\varphi$ is differentiable on $\argmin  \varphi$ and the gradient of $\varphi$ vanishes there. 
\end{enumerate}
\end{lemma}

Let $\varphi \in \Cper(\R^d)$ be semi-concave, then by Radmacher's theorem, it is differentiable almost everywhere. Denote
\[
	\cD(\varphi) = \{x \in \R^d \st \nabla \varphi(x) \text{ exists}\}. 
\]

The Lax-Oleinik operator regularizes continuous functions in the following sense. 
\begin{lemma}[Proposition 6.2.1 of \cite{Fathi2008}, see also Corollary 3.3 of \cite{KZ2017}]
If $\varphi \in \Cper(\R^d)$, then $T(\varphi)$ is $C$-semi-concave with $C$ depending only on $\|D^2 F\|_{C^0}$. 
\end{lemma}

The weak KAM theorem (see for example \cite{Fathi2008}) implies there exists a unique $c \in \R$ such that the operator $(T - c)$ admits a fixed point in $C(\T^d)$. Under our assumptions, the value $c = 0$ and the fixed point is unique after normalization. 

\begin{proposition}\label{prop:weak-KAM}
Suppose $\min_{x \in \T^d}F(x) = F(0) = 0$.  Then there is a unique $\psi \in \Cper(\R^d)$ satisfying $\psi(0) = 0$, such that
\begin{equation}
	\label{eq:lo-fixed}
	T(\psi) = \psi, \quad \text{i.e. }
	\psi(x) = \min_{y \in \R^d} \{\psi(y) + h(y, x)\}.  
\end{equation}
$\psi$ is $C-$semi-concave with $C>0$ depending only on $\|D^2 F\|_{C^0}$, and $\psi(x) > 0$ for all $x \ne 0$. 
\end{proposition}
\begin{proof}
The existence follows from the weak KAM theorem, see Theorem 4.1.1 of \cite{Fathi2008}. The uniqueness is due to the representation formula of weak KAM solutions, see for example Theorem 8.6.1 in \cite{Fathi2008}. 
\end{proof}

The Lax-Oleinik operator $T$ is closely related to the twist map $\Phi: \R^d \times \R^d \to \R^d \times \R^d$, defined by 
\begin{equation}
	\label{eq:gen-fun}
	\Phi(x_0,  p_0) = (x_1, p_1) \quad \Leftrightarrow \quad 
	v = - \partial_1 h(x_0, x_1), \quad p_1 = \partial_2 h(x_0, x_1). 
\end{equation}
Explicitly, 
\[
	\Phi(x_0, p_0) = (x_0 + p_0 + \nabla F(x_0) , p_0 + \nabla F(x_0)). 
\]
The map $\Phi$ also projects to a map on $\T^d \times \R^d$, which we denote by $\Phi_\T$.

Let us denote $h_\psi(y, x) = \psi(y) + h(y, x)$ for short. 
\begin{lemma}[See,  for example, Lemma 3.2 of \cite{KZ2017}]\label{lem:T-psi}
For all $x \in \R^d$ and $y \in \argmin h_\psi(\cdot, x)$, then $y \in \cD(\psi)$ and 
\[
\Phi(y, \nabla \psi(y)) = (x, p), 
\]
where $p$ is a super-gradient of $\psi$ at $x$. In particular, if $x \in \cD(\psi)$, then $y$ is the unique element in $\argmin h_\psi(\cdot, x)$ and $\Phi(y, \nabla \psi(y)) = (x, \nabla \psi(x))$. 
\end{lemma}

Let $\varphi$ be a semi-concave function on $\R^d$. Following \cite{Bernard08}, we define the (overlapping) \emph{pseudograph}
\[
	\cG_\varphi = \{(x, \nabla \varphi(x)) \st x \in \cD(\varphi)\}. 
\]
Lemma~\ref{lem:T-psi} can be rephrased in the pseudograph language as follows. 
\begin{lemma}[Proposition 2.7 of \cite{Bernard08}, see also Lemma 3.2 of \cite{KZ2017} in this setting] \label{lem:pseudograph}
\[
	\Phi^{-1}\left( \overline{\cG_\psi} \right) \subset \cG_\psi. 
\]
In particular, we have 
\[
	\overline{\cG_\psi} = 
	\left\{ (x,p) \st (x, p) = \Phi(y, \nabla \psi(y)) \text{ for some } y \in \argmin h_\psi(\cdot, x) \right\}. 
\]
\end{lemma}

If $x \in \cD(\psi)$, we denote by $\bary(x)$ the unique element of $\argmin h_\psi(\cdot, x)$. Then 
\[
\bary(x) = \pi_1 \Phi^{-1}(x, \nabla \psi(x)), 
\]
where $\pi_1: \R^d \times \R^d \to \R^d$ is the projection to the first component. Define 
\begin{equation}
	\label{eq:D-minus}
	\cD^-(\psi) = \pi_1  \Phi^{-1}\left( \overline{\cG_\psi} \right), 
\end{equation}
it follows from Lemma~\ref{lem:pseudograph} that 
\[
	\cD^-(\psi) = \bigcup_{x \in \R^d} \argmin h_\psi(\cdot, x) = \overline{\bary\left( \cD(\psi) \right)} \subset \cD(\psi). 
\]
Both $\cD(\psi)$ and $\cD^-(\psi)$ are periodic sets, and $\bary(x + l) = \bary(x) + l$ for any $l \in \Z^d$. We denote the projections of $\cD(\psi)$ and $\cD^-(\psi)$ by $\cD_\T(\psi)$ and $\cD_\T^-(\psi)$, while keeping the name $\bary$ unchanged. In the torus setting, 
\[
\{\bary(x)\} = \argmin_{y \in \T^d} \{ \psi(y) + A(y, x)\}, \quad x \in \cD_\T(\psi).
\]

While the function $\nabla \psi$ is only defined at almost every point, it is more regular on the set $\cD^-(\psi)$. This is described in Lemma~\ref{lem:D-2nd} and Corollary~\ref{cor:strong-lip}. 

\begin{lemma}\label{lem:D-2nd}
There exists $C > 0$ such that for every $y \in \cD^-(\psi)$ and $z \in \R^d$, 
\[
	|\psi(z) - \psi(y) - \nabla \psi(y) \cdot (z - y)| \le C |z - y|^2. 
\]
\end{lemma}
\begin{proof}
Since $\psi$ is $C$-semi-concave, we only need to prove the lower bound. 

If $y \in \cD^-(\psi)$, then there exists $x \in \R^d$, such that $y \in \argmin h_\psi(\cdot, x)$. By Lemma~\ref{lem:pseudograph} and \eqref{eq:gen-fun}, $\nabla \psi(y) = - \partial_1 h(y, x)$. Since $h(\cdot, x)$ is $C$-semi-concave, 
\[
	\begin{aligned}
	0 & \le h_\psi(z, x) - h_\psi(y, x) = \psi(z) - \psi(y) + h(z, x) - h(y, x) \\
	& \le \psi(z) - \psi(y) + \partial_1 h(y, x)  \cdot (z - y) + C |z - y|^2 \\
	& = \psi(z) - \psi(y) - \nabla \psi(y) \cdot (z - y) + C |z - y|^2, 
	\end{aligned}
\]
implying 
\[
	\psi(z) - \psi(y) - \nabla \psi(y) \cdot (z - y) \ge - C |z - y|^2. \qedhere
\]
\end{proof}

The following holds for general semi-concave functions.
\begin{lemma}\label{lem:strong-con}
Suppose $f: \R^d \to \R$ is $C$-semi-concave, and suppose for given $x \in \cD(f)$, there exists $C' > 0$ such that 
\[
	f(y) - f(x) - \nabla f(x) \cdot (y - x) \ge -  C' |y - x|^2, \quad \forall y \in \R^d. 
\]
Then for any $l_y \in \partial f(y)$, we have 
\[
	|l_y - \nabla f(x)| \le (4C' + 2C) |y - x|. 
\]
\end{lemma}
\begin{proof}
Consider any $y \in \R^d$ and $l_y \in \partial f(y)$. Set $w = (\nabla f(x) - l_y)/|\nabla f(x) - l_y|$, $\lambda = |y - x|$, we have 
\[
	f(y + \lambda w) - f(y) - l_y \cdot \lambda w  \le  C \lambda^2, 
\]
\[
	f(y + \lambda w) - f(x) - \nabla f(x) \cdot (y + \lambda w - x) \ge - C|y + \lambda w - x|^2 \ge - 4C' \lambda^2,
\]
the last inequality is due to $|y - x| = |\lambda w| = \lambda$. Subtract the two inequalities, we get 
\[
	\begin{aligned}
	- (4C' + C) \lambda^2 & \ge f(x) - f(y) + \nabla f(x) \cdot (y - x) + (\nabla f(x) - l_y) \cdot \lambda w \\
	& \ge - C\lambda^2 + \lambda |\nabla f(x) - l_y|, 
	\end{aligned}
\]
or $|\nabla f(x) - l_y| \le (4C' + 2C) |y - x|$. 
\end{proof}

It is known that $\nabla \psi$ is Lipschitz on $\cD^-(\psi)$ (see \cite{Fathi2008}), which is related to Mather's graph theorem. Our next statement is stronger, and follows directly from Lemma~\ref{lem:D-2nd} and \ref{lem:strong-con}. 
\begin{corollary}\label{cor:strong-lip}
For all $y \in \cD(\psi)$ and $x \in \cD^-(\psi)$, we have
\[
	|\nabla \psi(y) - \nabla \psi(x)| \le 6C |y - x|. 
\]
\end{corollary}
\begin{proof}
It follows from Lemma~\ref{lem:D-2nd} that Lemma~\ref{lem:strong-con} applies to $f = \psi$ and $C' = C$. 
\end{proof}

\section{Hyperbolicity and weak KAM solution}
\label{sec:hyp}

The map $\Phi_\T$ admits $(0,0)$ as a fixed point. We will show that it is hyperbolic and  $\cG_\psi$ locally coincides with the unstable manifold. Moreover, we show that $\psi$ is a Lyapunov function which is strictly contracted by the mapping $\bary$. 

\begin{proposition}
\label{prop:hyp}
There exists $C > 1$ depending only on $F$ such that the following hold. 
\begin{enumerate}[(1)]
	\item The fixed point $(0, 0)$ of $\Phi_\T$ is hyperbolic. Its local unstable manifold $W^u_\loc$ is $C^2$ smooth with $C^2$ norm bounded by $C$. The tangent plane to $W^u_\loc$ at $(0, 0)$ is given by the graph $\{(h, S^+ h) \st h \in \R^d\}$, where $S^+$ is a positive definite symmetric matrix. 
	\item There exists $r > 0$ such that for each $x \in B_r$, we have 
	\[
		\{(x, \nabla \psi(x)) \st x \in B_r \cap \cD_\T(\psi)\} = W^u_\loc \cap \{(x, v) \st x \in B_r\}. 
	\]
	We have $B_r \subset \cD^-(\psi) \subset \cD(\psi)$, and on $B_r$, $\psi$ is $C^3$ with uniformly bounded second derivatives. Moreover, $\psi(x) \ge C^{-1} |x|_\T^2$ and $\sqrt{\psi}$ is a $C$-Lipschitz function on $\T^d$. 
	\item We have 
	\[
		\{(x, \nabla\psi(x)) \st x \in \cD_\T(\psi)\} \subset W^u := \bigcup_{0 \le k < C} \Phi_\T^k W^u_\loc. 
	\]
	\item There exists $\kappa \in (0, 1)$ such that for each $x \in \cD(\psi)$,  $\psi(\bary(x)) \le \kappa^2 \psi(x)$. 
\end{enumerate}
\end{proposition}
\begin{proof}
(1):
The linearization of $\Phi_\T$ at $(0, 0)$ is given in block form by 
\[
	D\Phi(0, 0) = 
	\bmat{
	I_d + D^2 F(0) & I_d \\
	D^2 F(0) & I_d
	}. 
\]
To see this matrix is hyperbolic, denote $M = D^2 F(0)$, and suppose $S$ is a symmetric matrix such that the hyperplane $\{(h, Sh) \st h \in \R^d\}$ is invariant under $D\Phi(0, 0)$. Since
\[
	\bmat{I_d + M & I_d \\ M & I_d} \bmat{I_d \\ S} = \bmat{I_d + M + S \\ M + S}
\]
we get $M + S = S(I_d + M + S)$ or $S^2 + SM - M = 0$. The solutions are given by the quadratic formula $S^\pm = \frac12(- M \pm \sqrt{M^2 + 4M})$. The mapping $D\Phi(0, 0)$ takes any vector $(h, S^\pm h)$ to $(h_1, S^\pm h_1)$, where $h_1 = (I_d + M + S^\pm) h = (I_d + \frac12(M \pm \sqrt{M^2 + 4M}))h$. Since $M + \sqrt{M^2 + 4M}$ is positive definite, there exists $\kappa_0 \in (0, 1)$ such that 
\begin{equation}
	\label{eq:hor-expansion}
	|(I_d + M + S^+) h| > \kappa_0^{-1} |h|. 
\end{equation}
Moreover $(I_d + M + S^+)(I_d + M + S^-) = I_d$, so $|(I_d + M + S^-)h| < \kappa_0 |h|$. We have proven $(0, 0)$ is hyperbolic with an unstable bundle given by $(h, S^+h)$, where $S^+$ is positive definite. Also note for future use that $S^+$ commutes with $M$.  The rest of (1) are conclusions of standard hyperbolic theory. 

(2) and (3): (2) is proven in \cite{IS2009} in the flow setting and in \cite{KZ2017} for maps under much more general assumptions. We give a proof here for the sake of completeness. Suppose $x_0 \in \cD_\T(\psi) \subset \T^n$ and denote $v_0 = \nabla\psi(x_0)$, let $(x_{-n}, v_{-n}) = \Phi_\T^{-n}(x_0, v_0)$, we note that $\bar{y}(x_{-n}) = x_{-n-1}$ for all $n \ge 0$. 

Note that there is $C > 1$ such that
\[
	A(y, x) \ge C^{-1}(|y|^2_\T + |x|^2_\T), \quad \forall x, y \in \T^d. 
\]
Then for each $r > 0$, there exists $\delta > 0$ such that $A(y, x) > \delta$ unless $x, y \in B_r$. Given $n \in \N$, let $m(n, \delta) = \#\{k \in [-n, -1] \st |x_k|_\T > r\}$. Then 
\[
	\psi(x_0) = \psi(x_{-n}) +  \sum_{k = - n}^{-1} A(x_k, x_{k+1}) > \psi(x_{-n}) + m(n, \delta) \delta, 
\]
since $\psi$ is uniformly Lipschitz, we have $m(n, \delta) < \|\psi\|_{Lip}/\delta$. This argument shows that each backward orbit $\{x_{-n}\}_{n \in \N}$ can have at most finitely many points outside of $B_r$, i.e. $x_{-n} \to 0$ in $\T^d$. This convergence is uniform over all $x_0 \in \cD_\T(\psi)$ since the bound for $m(n, \delta)$ is independent of $x_0$. Moreover $v_{-n} = \nabla \psi(x_{-n}) \to 0$ since $\nabla \psi$ is Lipschitz over the set $\cD_\T^-(\psi)$. Hyperbolic theory implies any orbit backward asymptotic to an hyperbolic fixed point is contained in its (global) unstable manifold. Moreover, there exists $r_0 > 0$ such that if the entire orbit $(x_{-n}, v_{-n})_{n \in \N} \subset B_{r_0}((0, 0))$, then $(x_0, v_0) \in W^u_\loc \cap B_{r_0}(0, 0)$. This is the case if $x_0$ is sufficiently close to $0$. This proves 
\[
	\{(x, \psi(x)) \st x \in B_r \cap \cD_\T(\psi)\} = W^u_\loc \cap \{(x, v) \st x \in B_r\}
\]
and 
\[
	\{(x, \nabla\psi(x)) \st x \in \cD_\T(\psi)\} \subset \bigcup_{0 \le k < \infty} \Phi^k W^u_\loc. 
\]
The fact that the infinite union can be replaced with a finite one is again due to the uniform convergence of $(x_{-n}, v_{-n})$ to $(0, 0)$. 

Finally, $D^2 \psi(0) = S^+$ is positive definite. Since $0$ is the only global minimum of $\psi$, this implies that there exists $C > 0$ such that $\psi(x) \ge C^{-1} |x|_\T^2$. This also implies $\sqrt{\psi}$ is uniformly Lipschitz with Lipschitz constant depending only on $C$. 

(4): First we assume $x \in B_r(0)$. 
Since the graph of $\nabla \psi$ coincide with the smooth graph $W^u_\loc$ on $B_r$, the mapping $\bary$ is smooth on $B_r$, and $D \bary(0) = (I_d + M + S^+)^{-1}$ as we computed earlier. Since $\bary(x) = (I_d + M + S^+)^{-1} x + O(x^2)$, 
\[
	\psi(\bary)  = \langle  D^2 \psi(0) \bary, \bary \rangle + O(\bary^3) 
	= x^T (I_d + M + S^+)^{-T} S^+ (I_d + M + S^+)^{-1} x + O(x^3). 
\]
Since $(I_d + M + S^+)^{-1}$ is strictly contracting and commutes with $S^+$, by diagonalizing the matrices, there exists $\kappa_0 \in (0, 1)$ such that 
\[
	x^T (I_d + M + S^+)^{-T} S^+ (I_d + M + S^+)^{-1} x  < \kappa_0^2  x^T S^+ x^T, \quad 
	|(I_d + M + S^+)^{-1} x| < \kappa_0 |x|. 
\]
For $\kappa_1 \in (\kappa_0, 1)$, we can choose $r$ sufficiently small such that 
\[
	\psi(\bary(x))| < \kappa^2 \psi(x), \quad
	|\bary(x)| < \kappa |x|, 
	\quad x \in B_r. 
\]

If $x \notin B_r(0)$, then there exists $\delta > 0$ such that $A(y, x) \ge \delta$. Denote $\bary = \bary(x)$, we have 
\[
	\frac{\psi(\bary)}{\psi(x)} = \frac{\psi(\bary)}{\psi(\bary) + A(\bary, x)} 
	\le \frac{1}{1 + \delta/\|\psi\|_{C^0}} < 1. 
\]
It suffices to set $\kappa = \min \{\kappa_1, \sqrt{1/(1 + \delta/\|\psi\|_{C^0})}\}$. 
\end{proof}

Suppose $x \in \cD(\psi)$, then $\bary(x)$ reaches the unique minimum of $h_\psi(\cdot, x)$. We are interested in the non-degeneracy of this minimum, which turns out to be related to whether $\psi(x)$ has bounded second derivatives. Lemma~\ref{lem:D-2nd} ensures that this holds for every $x \in \cD^-(\psi)$. Using the fact that $\cG_\psi$ is contained in a smooth sub-manifold, we can also extend this estimate to a neighborhood. 

\begin{proposition}\label{prop:U-non-deg}
There exists an open set $U \supset \cD^-(\psi)$ and $\delta > 0$ depending only on $F$, such that the following holds. 
\begin{enumerate}
	\item $\psi$  is $C^3$ on $U$ with uniformly bounded second derivatives. 
	\item We have
	\[
		h_\psi(y, x) - h_\psi(\bary(x), x) \ge \delta |y - \bary(x)|^2, \quad
		\text{ for all } x \in U, \, y \in \R^d,
	\]
	and 
	\[
		h_\psi(y, x) - \min h_\psi(\cdot, x) \ge \delta, \quad 
		\text{ for all } y \notin U. 
	\]
\end{enumerate}
\end{proposition}

\begin{remark}
  The case $x\notin U, y\in U$ is not covered by this Proposition, and will be dealt with separately.
\end{remark}

We need a number of lemmas.

\begin{lemma}\label{lem:local-non-deg}
Suppose at some $x \in \cD(\psi)$, there exists $D_x > 0$ such that 
\[
	\left| \psi(z) - \psi(x) - \nabla \psi(x) \cdot (z - x) \right| \le D_x |z - x|^2, 
	\quad \forall z \in \R^d, 
\]
then for $\bary = \bary(x)$,  there exists $C, r > 0$ depending only on $F$, such that 
\[
	h_\psi(\bary + v, x) - h_\psi(\bary, x) \ge \frac{|v|^2}{8(D_x + C)}, 
	\quad  \text{ for all } |v| < r. 
\]
\end{lemma}
\begin{proof}
Write $\bary = \bary(x)$ and $D = D_x$. Then for $v, w \in \R^d$, 
\[
	\begin{aligned}
	& h_\psi(\bary + v, x) - h_\psi(\bary, x) \\
	& = \psi(\bary + v) + h(\bary + v, x) - \psi(\bary) - h(\bary, x) \\
	& = h(\bary + v, x) - h(\bary + v, x + w) + [\psi(\bary + v) + h(\bary + v, x + w) - \psi(\bary) - h(\bary, x)] \\
	& \ge h(\bary + v, x) - h(\bary + v, x + w) + \psi(x + w) - \psi(x) \\
	& = h(\bary + v, x) - h(\bary + v, x + w) + \partial_2 h(\bary + v, x)\cdot w \\
	& \qquad - (\partial_2 h(\bary + v, x) - \partial_2 h(\bary, x)) \cdot w 
	+ [\psi(x + w) - \psi(x) - \partial_2 h(\bary, x) \cdot w] \\
	& \ge - \frac12 \|\partial_{22} h\|_{C^0} |w|^2 - \partial_{12}h(\bary, x) v \cdot w 
	- \frac12 \|\partial_{112}h\| |v|^2 |w| - D|w|^2. 	
	\end{aligned}
\]
Set $w = - t v$, and note $\partial_{12} h = I_d$, $\|\partial_{22} h\|, \|\partial_{112} h\| \le C$, we get 
\[
	h_\psi(\bary + v, x) - h_\psi(\bary, x) \ge t |v|^2 (1 - (D + C) t + C |v|). 
\]
Set $t = \frac{1}{4(D + C)}$ and $|v| < \frac{1}{4C}$, we have 
\[
	h_\psi(\bary + v, x) - h_\psi(\bary, x) \ge \frac12 \lambda |v|^2 = \frac{1}{8(D + C)} |v|^2, \quad 
	\forall v < \frac{1}{4C}.  \qedhere
\]
\end{proof}

\begin{lemma}\label{lem:glob-lower-bound}
There exists $C, R > 0$ such that for all $x \in \R^d$ and $\bary \in \argmin h_\psi(\cdot, x)$, we have 
\[
	h_\psi(\bary + v, x) - h_\psi(\bary, x) \ge \frac14 |v|^2, \quad \text{ for all } |v| > R  
\]
and 
\[
	h_\psi(\bary + v, x) - h_\psi(\bary, x) \le C|v|^2, \quad \text{ for all } v \in \R^d. 
\]
\end{lemma}
\begin{proof}
The upper bound follows directly from semi-concavity, we only prove the lower bound. 

We first claim that there is a constant $C > 0$ such that $|\bary - x| < C$ for all $\bar{y} \in \argmin h_\psi(\cdot, x)$. Indeed, since $\psi$ is uniformly Lipschitz, $\overline{\cG_\psi}$ has uniformly bounded $p$ component. Therefore $\Phi^{-1}(\overline{\cG_\psi})$ is of bounded distance away from $\cG_\psi$, implying the claim. 

By periodicity, it suffices to prove our lemma for $x \in [-\frac12, \frac12)^d$. Since $|\bary - x| < C$, by resetting $C$ 
\[
	\begin{aligned}
	& h_\psi(\bary + v, x) - h_\psi(\bary, x) \ge \frac12 |\bary + v - x|^2 - \frac12 |\bary - x|^2 - 2 \|\psi\|_{C^0}-2\|F\|_{C^0} \\
	& \ge \frac12 |v|^2 - |v| |\bary - x| - 2\|\psi\|_{C^0} - 2 \|F\|_{C^0}
	\ge \frac12 |v|^2 - |v| C - 2 \|\psi\|_{C^0}-2\|F\|_{C^0}. 
	\end{aligned}
\]
Suppose $\frac18 R > C$ and $\frac18 R^2 > 2 \|\psi\|_{C^0}$, we get $h_\psi(\bary + v, x) - h(\bary, x) \ge \frac14 |v|^2$. 
\end{proof}

\begin{proof}[Proof of Proposition~\ref{prop:U-non-deg}]
Let $\cE \subset W^u \subset \T^d \times \R^d$ be the set of critical points for the projection $\pi_1|W: \T^d \times \R^d \to \T^d$. Then $\cE$ is a compact nowhere dense set of zero Lebesgue measure. We claim that $\cD^-(\psi) \cap \pi_1 \cE$ is empty. Otherwise, there must exists $x_0 \in \cD^-(\psi) \cap \pi_1 \cE$ and $ \cD(\psi) \ni x_k \to x_0$ such that $|\nabla \psi(x_k) - \nabla \psi(x_0)|/|x_k - x_0| \to \infty$ as $k \to \infty$, contradicting Corollary~\ref{cor:strong-lip}. 

This claim implies that the projection $\pi_1|W: W \to \T^d$ is regular at every point of $\cG^- : = \{(x, \nabla \psi(x)) \st x \in \cD^-(\psi)\}$. Implicit function theorem then implies that $\cG^-$ is contained in a $C^2$ smooth graph over an open neighborhood $U_1$ of $\cD^-(\psi)$. Using Corollary~\ref{cor:strong-lip} again, we conclude that for every $y \in \cD(\psi) \cap U_1$, $\nabla \psi$ must be contained in the same graph. We have now proven $\nabla \psi$ coincides with a $C^2$ function at almost every point in $U_1$, therefore $\psi$ must be $C^3$. The $C^2$  norm of $\psi$ is bounded as long as $\overline{U_1} \cap \pi_1 \cE = \emptyset$. The same hold if we lift $U_1$ to an open set in $\R^d$. This proves item (1) of our Proposition on the set $U_1$. 

Set $U = \{z \in U_1 \st \dist(z, \cD^-(\psi)) \le \frac12 \dist (\partial U_1, \cD^-(\psi))\}$, we claim that there exists $D > 0$ such that 
\begin{equation}
	\label{eq:U-2nd}
	\left| \psi(z) - \psi(y) - \nabla \psi(y) \cdot (z - y) \right| \le D |z - y|^2, 
	\quad \forall z \in \R^d, \, y \in U. 
\end{equation}
WE have proven that exists $D_1 > 0$ such that for all $y, x \in U_1$, 
\[
	|\nabla \psi(y) - \nabla \psi(x)| \le D_1 |y - x|. 
\]
If $z \in U_1$, \eqref{eq:U-2nd} holds since $\psi$ is $C^2$ with an uniform $C^2$ norm. Suppose $z \in \R^d \setminus U_1$. Then for each $y \in U$, there exists $x \in \cD^-(\psi)$ such that $|y - x| \le \frac12 |z - x|$, and hence $|z - y| \ge |z - x| - |y - x| \ge \frac12 |z - x|$. By Lemma~\ref{lem:D-2nd}, there exists $D_2 > 0$ such that 
\[
	|\psi(z) - \psi(x) - \nabla \psi(x) \cdot(z - x)| \le D_2 |z - x|^2. 
\]
Combine the two, we get 
\[
	\left| \psi(z) - \psi(y) - \nabla \psi(y) \cdot (z - y) \right| 
	\le D_2 |z - x|^2 + D_1 |y - x|^2 \le D_3 |z - y|^2
\]
for an absolute constant $D_3 > 0$. 

We now prove item (2). First, suppose $y \notin U$. 
\[
	\delta_1 = \min_{x \in \T^d, \, y \notin U} h_\psi(y, x) - h_\psi(\bary(x), x). 
\]
Then $\delta_1 > 0$ by compactness argument. If $|y - \bary(x)| \le R$ ($R$ is from Lemma~\ref{lem:glob-lower-bound}), we have
\[
	h_\psi(y, x) - h_\psi(\bary(x), x) \ge \delta \ge \frac{\delta_1}{R^2} |y - \bary(x)|^2. 
\]
If $|y - \bary(x)| > R$, Lemma~\ref{lem:glob-lower-bound} applies. 

Suppose $x \in U$. It follows from \eqref{eq:U-2nd} and Lemma~\ref{lem:local-non-deg} that there exists $r > 0$ and $\delta_2 > 0$ such that 
\[
	h_\psi(y, x) - h_\psi(\bary(x), x) \ge \delta_2 |y - \bary(x)|^2, \quad
	\text{ for all } x \in U, \, |y - \bary(x)| < r. 
\]
If $|y - \bary(x)| \ge r$, we set
\[
	\delta_3 = \min_{x \in \overline{U}, \, |y - \bary(x)| \ge r} h_\psi(y, x) - h_\psi(\bary(x), x) > 0, 
\]
and proceed in the same way as the $y \notin U$ case. 
\end{proof}

\section{The viscous equation via Hopf-Cole transformation}
\label{sec:hopf-cole}

Consider the  \eqref{eq:HJ-vis-ivp} with $T_0 = - n$, $T = 0$. We apply the Hopf-Cole transformation
\[
	u = \exp \left( -\frac{\varphi}{2\nu} \right), 
\]
which transform it to the inhomogeneous heat equation 
\begin{equation}
	\label{eq:heat}
	\begin{aligned}
	& u_t  = \nu \Delta u - \frac{1}{2\nu} F(x, t) u, & (x, t) \in \R^d \times (-n, 0)\\
	& u(x, -n) = u_0(x) =  e^{- \frac{1}{2\nu} \varphi_0(x)} & x \in \R^d.  
	\end{aligned}
\end{equation}

The solution to \eqref{eq:heat} is given by the Feynman-Kac formula, namely 
\[
	u(x, t) = \int dy \, u_0^\nu(y) \int \exp \left( - \frac{1}{2\nu} \int_0^t F(W (\tau)) d\tau \right) d \Pi_{x, y}^{(t,0)}(\nu; W),
\]
where $\Pi_{x,y}^{t,0}(\nu; \cdot)$ is the probability distribution of the Brownian motion $dx = \sqrt{2\nu} dW$ with the condition $W(0)= y$ and $W(t) = x$. 

For the kicked case, the formula for the solution is simplified. To solve from time $i$ to $i+1$, we multiply the function $u^-(\cdot, i)$ by the factor $e^{-\frac{1}{2\nu} F}$ to obtain $u^+(\cdot, i)$, representing the kicked force. We then solve the heat equation without force on the interval $(i, i+1)$ to obtain $u^-(\cdot, i+1)$. Formally, we have $u^-(\cdot, i+1) = \cL_\nu \left( u^-(\cdot, i) \right)$, where
\[
	\cL_\nu (u)(x) = \int K_\nu(y, x) u(y) dy, 
\]
and 
\[
	\begin{aligned}
	K_\nu(y, x) & = \frac{1}{(4\pi \nu)^{d/2}} \exp\left( -\frac{1}{2\nu} h(y, x) \right) \\
	& :=  \frac{1}{(4\pi \nu)^{d/2}} \exp\left( - \frac{1}{2\nu}\left( \frac12(y-x)^2 + F(y) \right) \right). 
	\end{aligned}
\]
Then for $n \in \N$, the solution $u_{-n}(x, 0)$ to \eqref{eq:heat} satisfy
\[
	\begin{aligned}
    & u_{-n}(x, 0)  = \cL_\nu^n (u_0)(x)  \\
	& \quad  = \idotsint u_0(x_{-n}) K_\nu(x_{-n}, x_{-n+1}) \cdots K_\nu(x_{-1}, x) d x_{-n} \cdots d x_{-1}. 	
	\end{aligned}
\]

The counterpart to Theorem~\ref{thm:unif-exp} in this setting is:
\begin{theorem}\label{thm:exp-conv-heat}
There exists $\nu_0 > 0$, $C > 0$ and $\lambda > 0$ depending only on $F$, such that the following hold for all  $\nu \in (0, \nu_0)$. There exists $u^\nu \in \Cper(\R^d)$ such that $0 \le \log u^\nu \le C/\nu$ and 
\[
\|\log\cL_\nu u^\nu - \log u^\nu\|_* = 0. 
\]
Moreover, for each $0 \le \log u_0, \log v_0 \le D/\nu$, we have 
\[
	\|\log \cL_\nu^n u_0 - \log \cL_\nu^n v_0\|_* \le e^{\frac{C  + D}{\nu}} e^{-\lambda n},\quad  
	\text{ for }  n \ge C/\nu. 
\]
In particular, $u^\nu$ is unique up to a constant. 
\end{theorem}

Theorem~\ref{thm:exp-conv-heat} is proven in Section~\ref{sec:boostrap}. We first show that our main theorem follows from Theorem~\ref{thm:unif-exp}. 

\begin{proof}[Proof of Theorem~\ref{thm:unif-exp}]
  We normalize $\varphi_0$ so that $\min \varphi_0 = 0$, and set $v_0 = e^{-\frac{1}{2\nu} \varphi_0}$, and  $\psi^\nu = - 2\nu \log u^\nu$ where $u^\nu$ is from Theorem~\ref{thm:exp-conv-heat}. Both $u^\nu$ and $v_0$ satisfy the assumptions of Theorem~\ref{thm:exp-conv-heat} with $D = \|\varphi_0\|_*$. It follows that 
\[
\|\varphi_n^\nu(\cdot, 0) - \psi^\nu(\cdot)\|_* = 2\nu \| \log \tL_\nu^n v_0 - \log \tL_\nu^n u^\nu\|_*
< e^{\frac{\|\varphi_0\|_* + C}{\nu}} e^{-\lambda n}
\]
Theorem~\ref{thm:unif-exp} follows. 
\end{proof}

\subsection{The conjugate kernel}

Let $\psi$ be the unique solution to \eqref{eq:lo-fixed}. Define
\[
	\tlh(y, x) = h(y, x) + \psi(y) - \psi(x).  
\]
The conjugate kernel is more convenient to study since $\argmin \tlh(\cdot, x) = \argmin h_\psi(\cdot, x)$, but in addition $\min \tlh(\cdot, x) = 0$ for all $x \in \R^d$. Define
\begin{equation}
	\label{eq:conj-L}
	\begin{aligned}
	\tlK_\nu(y, x) & = \exp \left( - \frac{1}{2\nu} \tlh(y, x) \right) = e^{\frac{1}{2\nu}\psi(x)} K_\nu(y, x)  e^{-\frac1{2\nu} \psi(y)}, \\
	\tL_\nu (u)(x) & = \int \tlK_\nu(y, x) u(y) dy = e^{\frac{1}{2\nu}\psi(x)} \cL_\nu \left( e^{-\frac{1}{2\nu}\psi} u \right). 
	\end{aligned}
\end{equation}
It's easy to see that $\tL_\nu^n (u) = e^{\frac{1}{2\nu}\psi} \cL_\nu^n \left( e^{-\frac{1}{2\nu}\psi} u \right)$. 

Proposition~\ref{prop:U-non-deg} implies:
\begin{corollary}\label{cor:tlh-lower}
\[
	\tlh(y, x) \ge \delta |y - \bary(x)|^2, \quad
	\text{ for all } x \in U, \, y \in \R^d,
\]
\[
	\tlh(y, x)  \ge \delta, \quad 
	\text{ for all } y \notin U. 
\]
\end{corollary}

\subsection{The Markov kernel}

We convert the kernels into Markov ones following \cite{Sinai1991}. Let $\bOne$ denote the constant function $1$, define:
\begin{equation}
	\label{eq:markov}
	\begin{aligned}
	& \pi_\nu^{(0)}(y, x)  = \frac{\tlK_\nu(y, x)}{\int \tlK_\nu(y, x) dy} = \frac{\tlK_\nu(y, x)}{\tL_\nu (\bOne) (x)}, \\
	& \pi_\nu^{(1)}(y, x)  =  \frac{\tlK_\nu(y, x)\tL_\nu(\bOne)(y)}{\tL_\nu^2(\bOne)(x)}, \\
	& \cdots  \\
	& \pi_\nu^{(n+1)}(y, x)  = \frac{\tlK_\nu(y, x) \tL_\nu^n (\bOne)(y)}{\tL_\nu^{n+1} (\bOne) (x)}, 
	\end{aligned}
\end{equation}
each kernel is Markov in the sense that 
\[
	\int \pi_n(y, x)  dy = \cP (\bOne)(x) = \bOne. 
\]
Define the Markov operators acting on functions
\[
	\cP_\nu^{(n)}(u)(x) = \int \pi_\nu^{(n)}(y, x) u(y) dy, \quad n \ge 0, \quad
	\cP_\nu^n(u) = \cP^{(n-1)}_\nu \cdots \cP^{(0)}_\nu u, 
\]
then
\begin{equation}
  \label{eq:telescope}
  \frac{\tL_\nu^n u(x)}{\tL_\nu^n \bOne(x)}
  = \int \pi_\nu^{(n-1)}(y, x) \frac{\tL_\nu^{n-1} u(y)}{\tL_\nu^{n-1} \bOne(y)} dy 
  = \cP_\nu^{(n-1)} \cdots \cP_\nu^{(0)} (u) = \cP_\nu^n(u). 
\end{equation}
The function $\tL^n \bOne $ is known as the \emph{partition function} in statistical mechanics. 

Let $U$ be as in Proposition~\ref{prop:U-non-deg}, define
\begin{equation}
	\label{eq:chi-nu}
	\chi_\nu(x) = \begin{cases}
	1 &  x \in U, \\
	\nu^{-\frac{d}{2}} & x \notin U. 
	\end{cases}  
\end{equation}
The following estimate of the partition function is a crucial technical step in our proof. 
\begin{proposition}\label{prop:laplace}
There exist $C > 1$, $\nu_0 > 0$, $Q_n > 0$ satisfying 
\[
	C^{-1} \le Q_{n+1} / Q_{n} \le C, 
\]
such that for $N_1(\nu) = C^{-1} (\nu \log \frac{1}{\nu})^{-\frac13}$,  
\[
	C^{-1} \le \frac{\tL_\nu^n (\bOne)(x)}{Q_n} \le C \chi_\nu(x),  \quad \text{for all} \quad 0 \le n \le N_1(\nu), \, 0 < \nu \le \nu_0.
\]
\end{proposition}
Proposition~\ref{prop:laplace} is proven by applying the classical Laplace method, but trying to obtain uniform estimates in $n$. The proof is postponed to the last two sections of this paper.

\section{Uniform contraction for the Markov operator}
\label{sec:hm}

\subsection{Lyapunov functions}

We describe the Lyapunov function approach to the convergence of Markov operators by Hairer and Mattingley (\cite{HM2011}). Let $\pi$ be a positive measurable function on $\R^d \times \R^d$ such that 
\[
	\cP(u)(x) = \int \pi(y, x) u(y) dy
\]
defines a bounded Markov operator from $L^\infty(\R^d)$ to $L^\infty(\R^d)$. 
Assume that: 
\begin{enumerate}
	\item (A1) There exists $V \in L^\infty(\R^d)$ with $V \ge 0$,  constants $M \ge 0$ and $\gamma \in (0, 1)$ such that 
	\[
		(\cP V)(x) \le \gamma V(x) + M, 
	\]
	for all $x \in X$. 
	\item (A2) There exists a constant $\alpha_0 \in (0, 1)$ and a probability density $g_0$ so that 
	\[
		\inf_{x : \, V(x) \le R} \pi(x, \cdot) \ge \alpha_0 g_0(\cdot), 
	\]
	where  $R > 2 M/(1-\gamma)$. 
\end{enumerate}

Given $\beta>0$, let us consider the norms
\begin{equation}
	\label{eq:norms}
	\|\varphi\|_{\beta V} = \sup_x \frac{|\varphi(x)|}{1 + \beta V(x)}, \quad 
	\|\varphi\|_{\beta V, *} = \inf_{C \in \R} \|\varphi + C\|_{\beta V}. 
\end{equation}
Choose the parameters as follows:
\begin{equation}
  \label{eq:hm-para}
  	\alpha_1 \in (0, \alpha_0), \,  \gamma_0 \in (\gamma+\frac{2M}{R}, 1), \, 
	\beta = \frac{\alpha_0}{M}, \,  
	\alpha = \max\left\{   1 - (\alpha_0 - \alpha_1), \,  \frac{2 + R \beta \gamma_0}{2 + R\beta} \right\}.
\end{equation}
\begin{theorem}[\cite{HM2011}]\label{thm:hm}
Suppose $\cP$ satisfies (A1) and (A2), and let $\beta, \alpha$ be chosen as described. Then 
\[
	\|\cP (u)\|_{\beta V, *} \le \alpha \|u\|_{\beta V, *}. 
\]
\end{theorem}

\subsection{Choice of Lyapunov functions for the Markov operators}

\begin{proposition}\label{prop:pi-hm}
For $D > 1$ and $\nu_0 > 0$, let $\pi_\nu(y, x)$, $\nu \in (0, \nu_0)$ be a family of a periodic Markov kernel on $\R^d \times \R^d$ such that
\[
\frac{1}{D\chi_\nu(x)} \le \frac{\pi_\nu(y, x)}{\tlK_\nu(y, x)} \le D \chi_\nu(y). 
\]
Let $\cP_\nu$ denote the Markov operator with kernel $\pi_\nu(y, x)$. 

Then there exists $C > 1$, $\alpha_0 \in (0, 1)$, $\nu_1 > 0$,  $R > 0$, and a family of probability density $(g_\nu)_{\nu \in (0, \nu_1)}$ on $\R^d$ depending only on $F$ and the constant $D$,  such that the following hold: 
\begin{enumerate}
	\item Let $V(x) = \psi(x) \chi^2_\nu(x)$, we have 
	\[
	(\cP_\nu V)(x) \le \gamma V(x) + C \nu. 
	\]
	\item $R > 2C/(1 - \gamma)$ and 
	\[
	\inf_{V(x) \le R\nu} \pi_\nu(x, y) \ge \alpha_0 g_\nu(y). 
	\]
\end{enumerate}
\end{proposition}

We now apply Proposition~\ref{prop:pi-hm} to the kernels $\pi^{(n)}(y, x)$ defined in \eqref{eq:markov}. 
\begin{corollary} \label{cor:contraction}
Suppose $\nu_0 > 0$, $N = N(\nu) \in \N$ where $\nu \in (0, \nu_0)$,  $D>1$ and $Q_n > 0$ are chosen such that the following hold:
\begin{equation}
  \label{eq:Qn-growth}
  	D^{-3} \le Q_n / Q_{n+1} \le D^3, \quad \text{ for all } 0 \le n \le N(\nu), 
\end{equation}
and  
\begin{equation}
  \label{eq:LnOne-bound}
  	D^{-1} \le \frac{\tL_\nu^n (\bOne)(x)}{Q_n} \le D \chi_\nu(x),  \quad \text{for all} \quad 0 \le n \le N(\nu), \, 0 < \nu < \nu_0.
\end{equation}
Then exists $C > 1$, $\alpha, \alpha_1 \in (0, 1)$, $\nu_1 > 0$ depending only $F$, $D$, such that for each $\nu \in (0, \nu_1)$ and $0 \le n \le N(\nu)$, we have 
\[
	\|\cP_\nu^{(n)} u\|_{\beta V, *} \le \alpha \|u\|_{\beta V, *}, 
\]
where $\beta = \frac{\alpha_1}{C\nu}$ and $V(x) = \psi(x) \chi_\nu^2(x)$.
\end{corollary}
\begin{proof}[Proof of Corollary~\ref{cor:contraction}]
For each $0 \le n \le N(\nu)$, we have $\pi_\nu^{(n)}(y, x)/\tlK_\nu(y, x) = \tL_\nu^n\bOne(y)/\tL_\nu^{n-1}\bOne(x)$. It follows from \eqref{eq:Qn-growth} and \eqref{eq:LnOne-bound} that
\[
\frac{1}{D^4\chi_\nu(x)} \le \frac{\pi_\nu(y, x)}{\tlK_\nu(y, x)} \le D^4 \chi_\nu(y), 
\]
hence Proposition~\ref{prop:pi-hm} applies with $D$ replaced with $D^4$. We choose parameters according to \eqref{eq:hm-para} and the statement follows. 
\end{proof}

\begin{proof}[Proof of Proposition~\ref{prop:pi-hm}]

\textbf{Item (1), case 1}: $x \in U$. In this case we have 
\[
	(\cP V)(x) = \int \pi_\nu(y, x) V(y) dy \le D \int \tlK(y, x) V(y) dy. 
\]
By Proposition~\ref{prop:hyp}, $\sqrt{\psi}$ is a $C$-Lipschitz function. Then for any $\epsilon > 0$, and $x \in U$, we have 
\[
	\begin{aligned}
	{\psi}(y) & \le \left( \sqrt{\psi}(\bary(x)) + C|y -\bary(x)| \right)^2 \\
	& \le \psi(\bary(x)) + C^2 |y - \bary(x)|^2 + 2 C \sqrt{\psi}(\bary(x)) |y -\bary(x)| \\
	& \le (1 + \epsilon^2) \psi(\bary(x)) + C^2(1 + \epsilon^{-2}) |y - \bary(x)|^2 \\
	& \le (1 + \epsilon^2) \kappa^2 \psi(x) + C^2(1 + \epsilon^{-2}) |y - \bary(x)|^2  \\
	& \le \gamma \psi(x) + C_1 |y - \bary(x)|^2,
	\end{aligned}
\]
where we have set $\gamma = (1 + \epsilon^2) \kappa^2$, $C_1 = C^2(1 + \epsilon^{-2})$. We then choose $\epsilon$ such that $\gamma \in (0, 1)$. 

Write $\bary = \bary(x)$ for short, and let $r > 0$ be such that $B_r(\bary) \subset U$, then
\[
	(\cP V)(x) \le \int_{B_r(\bary)} \pi_\nu(y, x) V(y) dy +  D \int_{\R^d \setminus B_r(\bary)}  \tlK(y, x) V(y) dy.  
\]
We have 
\[
	\begin{aligned}
	& \int_{B_r(\bary)} \pi_\nu(y, x) V(y) dy = \int_{B_r(\bary)} \pi_\nu(y, x) \psi(y) dy\\
	& \le C_1D \int_{B_r(\bary)} \tlK_\nu(y, x) |y - \bary|^2 dy + \gamma \psi(x) \int_{B_r(\bary)} \pi_\nu(x, y) dy \\
	& = C_1D (4\pi)^{-\frac{d}{2}} \nu^{-\frac{d}{2}} \int_{B_r(\bary)} e^{-\frac{1}{2\nu}\tilde{h}(y, x)} |y - \bary|^2 dy
	+ \gamma \psi(x). 
	\end{aligned}
\]
By Corollary~\ref{cor:tlh-lower}, 
\[
	\begin{aligned}
	& \nu^{-\frac{d}{2}} \int_{B_r(\bary)} e^{-\frac{1}{2\nu}\tilde{h}(y, x)} |y - \bary|^2 dy
	\le \nu^{-\frac{d}{2}} \int_{B_r(\bary)} e^{-\frac{1}{2C\nu}|y- \bary|^2} |y - \bary|^2 dy \\
	& = \nu^{-\frac{d}{2}} \int_{B_r(0)} e^{-\frac{1}{2C\nu}|v|^2} |v|^2 dv 
	= \nu \int_{B_{r/\sqrt{\nu}}} e^{-\frac{1}{C}|v|^2} |v|^2 dv \le C_2 \nu, 
	\end{aligned}
\]
where $C_2 = \int_{\R^d} e^{-\frac{1}{C}|v|^2} |v|^2 dv$. 

On the other hand, since $V(x) \le C \chi_\nu^2(x) \le C \nu^{-d}$, we have 
\[
	\begin{aligned}
	& \int_{\R^d \setminus B_r(\bary)} \tlK_\nu(y, x) V(y) dy \\
	& \le (4\pi)^{\frac{d}{2}} C \nu^{-\frac{3d}{2}} 
	\int_{\R^d \setminus B_r(\bary)} e^{-\frac{1}{2\nu} \tlh(y, x)} dy 
	\le C_3 \nu^{-\frac{3d}{2}} \int_{\R^d \setminus B_r(\bary)} e^{-\frac{1}{2C\nu} |y - \bary|^2} dy \\
	& =  C_3 \nu^{-\frac{3d}{2}} \int_{|v| \ge r} e^{-\frac{1}{2C\nu}|v|^2} dv 
	\le C_3 \nu^{-\frac{3d}{2}} e^{-(\frac{1}{2C\nu} - 1)r^2} \int_{|v| \ge r} e^{-\frac{1}{2C}|v|^2} dv \\
	& \le C_3 e^{r^2} \nu^{-\frac{3d}{2}} e^{-\frac{1}{2C\nu}r^2}.  
	\end{aligned}
\]
We now choose $\nu_1$ sufficiently small such that $C_3 e^{r^2} \nu^{-\frac{3d}{2}} e^{-\frac{1}{2C\nu}r^2} < \nu$. 

Combine all the estimate, for $x \in U$, there is $C_4 > 1$ depending on $F$ and $D$ such that 
\begin{equation}
  \label{eq:pv-case1}
  	(\cP V)(x) \le C_4 \nu + \gamma \psi(x)  = C_4 \nu + \gamma V(x). 
\end{equation}

\textbf{Item (1), case 2}: $x \notin U$.  Then 
\[
	\int \pi_\nu(y, x) V(y) dy \le D \int \tlK_\nu(y, x) \psi(y) \chi_\nu^3(y) dy. 
\]
On one hand, 
\[
	\begin{aligned}
	\int_U \tlK_\nu(y, x) \psi(y) \chi_\nu^3(y) dy 
	& = \int_U \tlK_\nu(y, x) \psi (y) dy 
	\le C \int_U \tlK_\nu(y, x) dy \\
	& \le C \tL_\nu \bOne(x)  \le C \nu^{-\frac{d}{2}}, 
	\end{aligned}
\]
where we used Proposition~\ref{prop:laplace}. On the other hand, 
\[
	\begin{aligned}
	& \int_{U^c} \tlK_\nu(y, x) \psi(y) \chi_\nu^3(y) dy 
	\le C \nu^{-2d} \int_{U^c} e^{-\frac{1}{2\nu} \tlh(y, x)} dy \\
	& \le  C \nu^{-2d} e^{-(\frac{1}{2\nu}- 1)} \int_{U^c} e^{- \tlh(y, x)} dy 
	\le C \nu^{-2d} e^{-\frac{1}{2\nu}}. 
	\end{aligned}
\]
By choosing $\nu_1$ small enough, we can ensure $C \nu^{-2d} e^{-\frac{1}{2\nu}} < C\lambda \nu^{-\frac{d}{2}}$. Therefore for $x \notin U$, we have 
\[
	(\cP V)(x) \le 2 C \nu^{-\frac{d}{2}}\leq 2C^2\nu^{d/2}\cdot V(x). 
\]
By choosing $\nu_1$ again, we can ensure $(2 C^2 \nu^{\frac{d}{2}}) < \gamma$. 

Combine the two cases, we have proved item (1).

\textbf{Item (2): } Note that 
\[
	\pi_\nu(y, x) \ge \frac{1}{D\chi_\nu(x)} \tlK_\nu(y, x) = D^{-1} \tlK_\nu(y, x). 
\]
Since $\pi_\nu$ is periodic, it suffices to prove (2) for $x \in [-\frac12, \frac12)^d$. Note that there exists $C > 1$  such that $V(x) \ge \psi(x) \ge C^{-1} |x|_\T^2$. It follows that for $R \nu_0$ sufficiently small, for all $\nu \in (0, \nu_0)$, 
\[
	\{x \st V(x) < R\nu\} \subset \{x \st |x|_{\T^2} < C \sqrt{R\nu}\} = \{x \st |x| < C \sqrt{R\nu}\} = B_{C\sqrt{R\nu}} \subset U.  
\]
Proposition~\ref{prop:hyp} implies $\bary(x) \in B_{C\sqrt{R\nu}}$ for all $x \in B_{C\sqrt{R\nu}}$. By Corollary~\ref{cor:tlh-lower}, there exists $C > 1$ such that 
\[
\tlh(y, x) \le C|y - \bary(x)|^2, \quad y \in \R^d. 
\]
It follows that 
\[
\begin{aligned}
 	\inf_{V(x) < R\nu} \pi_\nu(y, x) & 	\ge D^{-1} (4\pi \nu)^{-d/2} \inf_{|x| < C\sqrt{R\nu}} e^{-\frac{C}{2\nu}|y - \bary(x)|^2}  \\
 	& \ge D^{-1} (4\pi \nu)^{-d/2} \inf_{|z| < C\sqrt{R\nu}} e^{-\frac{C}{2\nu}|y - z|^2} := G_\nu(y). 
\end{aligned}
\]
Set 
\[
\alpha_\nu = \int G_\nu(y) dy, \quad g_\nu(y) = \alpha^{-1}_\nu G_\nu, 
\]
then $g_\nu$ is a probability density satisfying (2). It suffices to prove that there exists $\alpha_0 \in (0, 1)$ such that $\alpha_\nu \ge \alpha_0$. 

Indeed, 
\[
\begin{aligned}
\int G_\nu(y) dy & = (4\pi \nu)^{-d/2} \int \inf_{|z| < C\sqrt{R\nu}} e^{-\frac{C}{4\nu}} |y - z|^2 \\
& = (4\pi)^{-d/2} \int \inf_{|\zeta| < C \sqrt{R}} e^{- \frac{C}{2} |y - \zeta|^2} := \alpha_0
\end{aligned}
\]
is independent of $\nu$. Moreover, since $C > 1$, 
\[
\alpha_0 \le (4\pi)^{-d/2} \int e^{-\frac{1}{2}|y - x|^2} dy < 1. 
\]

To choose parameters, we only need to choose $R$ sufficiently large such that $R > 2C_4/(1 - \gamma)$ ($C_4$ is from \eqref{eq:pv-case1}) , then choose $\alpha$ and $\nu_0$ depending on $R$. 
\end{proof}

\section{Bootstrap argument}
\label{sec:boostrap}

By Proposition~\ref{prop:laplace}, there exists $C > 1$, $\nu_0 > 0$ and $Q_n > 0$ satisfying
\begin{equation}
  \label{eq:Cn-growth}
  	C^{-1} \le \frac{Q_{n+1}}{Q_n} \le C, 
\end{equation}
such that for $n \le N_1(\nu) = C^{-1} (\nu \log \frac{1}{\nu})^{-\frac13}$ and $\nu\in (0,v_0)$, we have
\begin{equation}
	\label{eq:Ln-one}
	C^{-1} \le \frac{\tL_\nu^n \bOne(x)}{Q_n} \le C \chi_\nu(x). 
\end{equation}
This estimate allows us to apply Corollary~\ref{cor:contraction} up to $n = N_1(\nu)$. Using this corollary, we would like to bootstrap the estimate \eqref{eq:Ln-one} to arbitrary $n$. More precisely, we will prove the following:

\begin{theorem}
\label{thm:bootstrap}
There exist constants $M > 0$ and $\nu_0> 0$, and $Q_n > 0$ depending only on $F$, such that
\[
	M^{-3} \le \frac{Q_{n+1}}{Q_n} \le M^3, 
\]
and
\[
	M^{-1} \le \frac{\tL_\nu^n \bOne(x)}{Q_n} \le M \chi_\nu(x)
\] 
hold for any $n \in \N$ and $\nu \in (0, \nu_0)$. 
\end{theorem}

We first show that Theorem~\ref{thm:bootstrap} implies Theorem~\ref{thm:exp-conv-heat}, hence our main theorem. 

By Theorem~\ref{thm:bootstrap}, Corollary~\ref{cor:contraction} applies to all $n \in \N$ with $D = M$. We conclude that there exists $C > 1$, $\alpha,\alpha_1 \in (0, 1)$ depending on $F$ and $M$, such that for $\beta =\alpha_1 (C\nu)^{-1}$ and all $n \in \N$,
\begin{equation}
	\label{eq:P-contraction}
	\|\cP^{(n) }\|_{\beta V, *} \le \alpha^n \|u\|_{\beta V, *}. 
\end{equation}
We state two lemmas on norm estimates. 

\begin{lemma}
\label{lem:beta_estimate}
Let $\beta = \alpha_1/(D_1\nu)$ for $\alpha_1 \in (0, 1)$ and $V = \psi \chi_\nu^2$. Then there exists $C > 1$ depending only on $F$ such that:
\begin{equation}
	\label{eq:7}
	\normbetastar{u}\leq \normstar{u}.
\end{equation}
and
\begin{equation}
	\label{eq:9}
	\normstar{u}\leq C D_1 \nu^{-d-1} \normbetastar{u}
\end{equation}
\end{lemma}
\begin{proof}
Let $\bar{u}$ be the constant such that $\normstar{u}=\norm{u-\bar{u}}$ and let $\bar{u}_{\beta}$ be the constant such that
$\normbetastar{u}=\norm{u-\bar{u}_{\beta}}_{\beta V}$. Then,
\begin{equation}
	\label{eq:8}
	\begin{split}
	\normbetastar{u}&=\norm{\frac{u-\bar{u}_{\beta}}{1+\beta V}}\leq \norm{\frac{u-\bar{u}}{1+\beta V}}\leq \norm{u-\bar{u}}=\normstar{u}.
	\end{split} 
\end{equation}
Similarly,
\begin{equation}
	\label{eq:1}
	\begin{split}
	\normstar{u}&\leq \norm{u-\bar{u}_{\beta}}\leq (1+\beta \|V\|_{C^0} )\norm{u-\bar{u}_{\beta}}_{\beta V}\\
	& \le \|\psi\|_{C^0} D_1 \nu^{-d-1} \normbetastar{u}, 
	\end{split}
\end{equation}
since $\|V\|_{C^0} \le \|\psi\|_{C^0} \nu^{-d}$, and $\beta \le D_1^{-1} \nu^{-1}$. 
\end{proof}

\begin{lemma}\label{lem:ratio-star}
Suppose $u, v \in \Cper(\R^d)$ with $\min u = a$, $\min v = b$, $a, b \ge 1$. Suppose $\omega := \max\{\|u\|_*, \|v\|_*\} < \frac{1}{4}$, 
\[
	\left\| \log \frac{u}{v} \right\|_* \le  4\omega.
\]
\end{lemma}
\begin{proof}
Note that $0 \le u - a \le 2 \|u\|_*$, and $0 \le v - b \le 2 \|v\|_*$. We have
\[
	\frac{u}{v} - \frac{a}{b} = \frac{u(b - v) + (u - a)v}{v b} 
	\le \frac{u - a}{b} \le \frac{2\|u\|_*}{b} \le \frac{a}{b}\cdot 2\|u\|_*, 
\]
by the same calculation, 
\[
	\frac{u}{v} - \frac{a}{b} \ge - \frac{a}{b} \cdot 2 \|v\|_*.
\]
We get
\[
\left| \frac{u}{v} \Bigr/ \frac{a}{b} - 1\right| < 2\omega < \frac12. 
\]
Note that $|\log(1 + x)| \le 2|x|$ for all $|x| < \frac12$, we get 
\[
\left\| \log \frac{u}{v} \right\|_* 
\le  \left\|\log \left( \frac{u}{v} \Bigr / \frac{a}{b} \right) \right\|
\le 2 \left\| \frac{u}{v} \Bigr / \frac{a}{b} - 1 \right\| \le 4\omega. 
\]
\end{proof}

\begin{proof}[Proof of Theorem~\ref{thm:exp-conv-heat}]
Theorem~\ref{thm:bootstrap} implies for all $n \in \N$, and $\nu \in (0, \nu_0)$,
\[
\left\| -\frac{\psi}{2\nu} + \log \tL_\nu^n \left( e^{\frac{\psi}{2\nu}}  \right) \right\|_*
= 
\|\log \tL^n_\nu \bOne\|_* \le \log(M^2 Q_n) + \frac{d}{2} \log \frac{1}{\nu},  
\]
where we applied \eqref{eq:conj-L} in the first equality. 
It follows that 
\[
\left\| \log \tL_\nu^n \left( e^{\frac{\psi}{2\nu}}  \right) \right\|_* \le \log(M^2 Q_n) + \frac{d}{2} \log \frac{1}{\nu} + \|\psi\|_*/(2\nu) \le C/\nu
\]
if $C > \|\psi\|_*/2 + 1$ and $\nu_0$ is small enough. For a \emph{fixed} $\nu > 0$, the functions $\log \tL_\nu^n \left( e^{\frac{\psi}{2\nu}} \right)$ are uniformly (in $n$) Lipschitz (see for example \cite{GIK+2005}), and therefore $\log \tL_\nu^n \left( e^{\frac{\psi}{2\nu}} \right)$ has a limit point in $\|\cdot\|_*$, which we call $u^\nu$ and normalize to $\min \log u^\nu = 0$.

Suppose $0 \le \log u \le D/\nu$ for some $D > 0$. By \eqref{eq:P-contraction} and Lemma~\ref{lem:beta_estimate}, we have 
\begin{equation}
	\label{eq:Lu-exp-bound}
	\begin{aligned}
	\|\tL_\nu^n u / (\tL_\nu^n \bOne)\|_*  
	& = \|\cP^n u\|_* \le C\nu^{-d-1} \|\cP^n u\|_{\beta V, *} \le C \alpha^n \nu^{-d-1} \|u\|_{\beta V, *} \\
	& \le C \alpha^n \nu^{-d-1} \|u\|_*  \le C \alpha^n \nu^{-d-1} e^{D/\nu}. 
	\end{aligned}
\end{equation}
Note also $\tL_\nu^n u / (\tL_\nu^n \bOne) = \cP^n u \ge \cP^n \bOne = 1$.

We have
\[
	\cL_\nu^n u_0 = e^{-\frac{\psi}{2\nu}} \tL_\nu^n \left( e^{\frac{\psi}{2\nu}} u_0 \right). 
\]
Suppose $0 \le \log u_0, \log v_0 \le D_1/\nu$ for some $D_1 > 0$, we set $u = e^{\psi/(2\nu)}u_0$, $v = e^{\psi/2\nu} v_0$, then $0 \le \log u, \log v \le (D_1 + \|\psi\|_{C^0})/\nu$. Denote
\[
\omega_n = C \alpha^n \nu^{-d-1} e^{(D_1 + \|\psi\|_{C^0})/\nu}, 
\]
then there exists $C_1> 0$ such that if $n > C_1/\nu$, $\omega_n < \frac{1}{4}$. Then 
\[
	\begin{aligned}
	\left\| \log \frac{\cL_\nu^n u_0}{\cL_\nu^n v_0} \right\|_*
	= \left\|\log  \frac{\tL_\nu^n u}{\tL_\nu^n v} \right\|_* 
	= \left\| \log \frac{\tL_\nu^n u / \cL_\nu^n \bOne}{\tL_\nu^n v / \tL_\nu^n \bOne} \right\|_*
	\le 4\omega_n
	\end{aligned}
\]
where we used \eqref{eq:Lu-exp-bound} and Lemma~\ref{lem:ratio-star}. By choosing a larger $C_1 > \|\psi\|_{C^0}$ if needed, we can ensure $4\omega_n \le e^{(D_1 + C_1)/\nu} \alpha^n$ for all $n \ge C_1/\nu$ and $\nu \in (0, \nu_0)$.  
\end{proof}

\begin{lemma}\label{lem:tL-chi}
There exists a constant $C > 1$ such that 
\[
\tL_\nu \chi_\nu \le C \chi_\nu. 
\]
\end{lemma}
\begin{proof}
The case $x \in U$ follows directly from Proposition~\ref{prop:U-laplace} for $n = 1$. The case $x \notin U$ is identical to the proof of Item (1), case 2 of Proposition~\ref{prop:pi-hm}. 
\end{proof}

\begin{proof}[Proof of Theorem~\ref{thm:bootstrap}]
Let $C$ be the largest of the constants in \eqref{eq:Cn-growth}, \eqref{eq:Ln-one}, Lemma~\ref{lem:tL-chi} and Lemma~\ref{lem:beta_estimate}. Set $M = 2C$, and let $C_1 > 1$, $\alpha, \alpha_1 \in (0, 1)$ and $\nu_1$ be the constants obtained by applying Corollary~\ref{cor:contraction} with parameter $D = M$. Choose $0 < \nu_2 \le \nu_1$ such that 
\begin{equation}
  \label{eq:bounded-by-C}
  2 C C_1 M^2 \nu^{-3d/2-1} \alpha^{N_1(\nu)} < 1 \quad \text{ for all }  \nu \in (0, \nu_2). 
\end{equation}
(This is possible because $N_1(\nu) = C^{-1} (\nu \log \frac{1}{\nu})^{-\frac13}$). 

First we show that there exists $Q_n > 0$ such that for all $n \in \N$ and $\nu \in (0, \nu_2)$. 
\begin{equation}
  \label{eq:Ln-M-bound}
  M^{-1} < \frac{\tL_\nu^n \bOne}{Q_n} < M \chi_\nu. 
\end{equation}

Fix a $\nu \in (0, \nu_2)$, denote $N = N_1(\nu)$, we proceed by induction in step size $N$. Suppose that \eqref{eq:Ln-M-bound} hold for $0 \le n \le kN$ for a given $k \ge 1$. The inductive hypothesis holds for $k = 1$ by Proposition~\ref{prop:laplace}.  Corollary~\ref{cor:contraction} implies for all $0 \le n \le k N$, 
\[
\|\cP^n u\|_{\beta V, *} \le \alpha^n \|u\|_{\beta V, *}, 
\]
where $\beta = \alpha_1/(C_1\nu)$. Set $R_n = \min \tL_\nu^n \bOne$, then if \eqref{eq:Ln-M-bound} is satisfied, we have 
\begin{equation}
  \label{eq:minL-bound}
  \left\| \tL_\nu^n \bOne / R_n \right\|_* \le \frac{\sup \tL_\nu^n \bOne}{ \min \tL_\nu^n \bOne} 
  \le M^2 \nu^{-d/2}. 
\end{equation}

Suppose $kN < n \le (k+1) N$, we have 
\[
\begin{aligned}
\tL_\nu^n \bOne 
& = (\tL_\nu^N \bOne)  \cP^N(\tL_\nu^{n - N}\bOne)  
\le C \chi_\nu  Q_N R_{n - N} \cP^N\left( \tL_\nu^{n - N}\bOne/ R_{n - N} \right) \\
& = C \chi_\nu  Q_N R_{n - N}  \left( 1 + 2 \left\| \cP^N \left(  \tL_\nu^{n - N}\bOne/ R_{n - N} \right) \right\|_* \right) \\
& \le C \chi_\nu  Q_N R_{n - N}  \left( 1 + 2 C C_1 \nu^{-d-1} \left\| \cP^N \left(  \tL_\nu^{n - N}\bOne/ R_{n - N} \right) \right\|_{\beta V, *} \right) \\
& \le C \chi_\nu  Q_N R_{n - N}  \left( 1 + 2 C C_1 \nu^{-d-1} \alpha^N \|\tL_\nu^{n - N}\bOne/ R_{n - N}\|_{*} \right) \\
& \le C \chi_\nu Q_N R_{n - N} \left( 1 + 2 C C_1 M^2 \nu^{-3d/2-1} \alpha^N \right) \le
2C \chi_\nu Q_N R_{n - N},  
 \end{aligned}
\]
where in the last line we used \eqref{eq:bounded-by-C}. 
The converse is easier since 
\[
\tL_\nu^n \bOne = (\tL_\nu^N \bOne)  \cP^N(\tL_\nu^{n - N}\bOne) \ge C^{-1} Q_N R_{n - N}. 
\]
Therefore \eqref{eq:Ln-M-bound} holds with $Q_n = Q_N R_{n - N}$. 

We now show that the constants chosen satisfies
\[
M^{-3} \le \frac{Q_{n+1}}{Q_n} \le M^3. 
\]
Byy Lemma \ref{lem:tL-chi}
\[
\tL_\nu^{n+1} \bOne = \tL_\nu \left( \tL_\nu^n \bOne \right) 
\le \tL_\nu \left( M Q_n \chi_\nu \right) \le C M Q_n \chi_\nu,
\]
For the lower bound, 
\[
\tL_\nu^{n+1} \bOne = \tL_\nu \left( \tL_\nu^n \bOne \right) \ge M^{-1}Q_n\tL_\nu\bOne\ge C^{-1} M^{-1} Q_n. 
\]
We now use \eqref{eq:Ln-M-bound} to get 
\[
M^{-1} \le \tL_\nu^{n+1} \bOne/Q_{n+1} \le CM\chi_\nu\frac{Q_n}{Q_{n+1}}, \quad
M\chi_\nu \ge \tL_\nu^{n+1} \bOne/Q_{n+1} \ge C^{-1} M ^{-1} \frac{Q_n}{Q_{n+1}}, 
\]
so
\[
C^{-1}M^{-2}\leq \frac{Q_n}{Q_{n+1}}\leq CM^2
\]
\end{proof}

We have concluded the proof of the main theorem with the exception of Proposition~\ref{prop:laplace}, which we prove in the next two sections. 

\section{Estimate of Hessian matrix}
\label{sec:hess}

Fix $x \in \cD(\psi)$ and $n\in\N$. Let $X=(x_{-n},\ldots,x_{-1}) \in \R^{nd}$, $x_0 = x \in \R^d$, and denote $H_{n,x}(X)=\sum_{i=-n}^{-1}h(x_i,x_{i+1})+\psi(x_{-n})-\psi(x)$, then 
\begin{equation}
	\label{eq:tL-int}
	\tL_\nu^n \bOne(x) = (4\pi \nu)^{-\frac{nd}{2}} \int\cdots\int\exp\left(-\frac{1}{2\nu}H_{n,x}(X)\right)dx_{-n}\cdots dx_{-1}. 
\end{equation}
The classical Laplace method (see for example \cite{dBru1981}) suggests that if the function $H_{n, x}(X)$ has a unique global minimum at $X_*$, then 
\[
	\cL_\nu^n \bOne (x) \sim e^{- \frac{1}{2\nu} H_{n, x}(X_*)} \left( \det D^2 H_{n, x}(X_*) \right)^{-\frac12} ( 1 + o_{\nu \to 0}(1)). 
\]
In this section we carry out preliminary estimates on the Hessian matrix $D^2 H_{n, x}(X_*)$.

Set $x_{-k}^* = x^*_{-k}(x) = \pi_1 \Phi^{-k}(x, \nabla \psi(x))$ for all $k < 0$, then $H_{n, x}$ reaches its global minimum at $X^* = X^*(x) = (x_{-n}^*, \ldots, x_{-1}^*)$.  Denote 
\[
	\cA_{n, x}(X) = \cA_n(x_{-n},\cdots,x_{-1}) = D^2_{x_{-n},\ldots,x_{-1}}H_{n,x}(X), \quad
	\cA_{n, x}^* = \cA_{n,x}(X^*(x)),  
\]
we have
\begin{equation}
	\label{eq:hess}
	\begin{aligned}
	& \qquad \cA_{n,x}(x_{-n}, \cdots, x_{-1})  = \\
	& \begin{bmatrix}
	I_d +D^2(F + \psi)(x_{-n}) & - I_d                                      &        &                         & \\
	-I_d                       & 2I_d + D^2 F(x_{-(n-1)})                  & \ddots &                         & \\
	& \ddots                                   & \ddots &                         & \\
	&                                          &        & 2I_d + D^2 F(x_{-2}) & -I_d \\
	&                                          &       & -I_d                       &  2I_d + D^2 F(x_{-1})
	\end{bmatrix}. 
	\end{aligned}
\end{equation}

The goal of this section is to prove the following estimates:
\begin{proposition}\label{prop:Hessian-bound}
There exists $C, \mu>1$, independent of $n$, such that for all $x \in U$ and $n\in\N$, 
\begin{equation}
	\label{eq:hessian-bound}
	C^{-1} \mu^n\leq \det \cA_{n,x}^{*}\leq C \mu^n
\end{equation}
\end{proposition}
\begin{remark}
The exponent $\mu > 1$ is related to the hyperbolic fixed point $(0, 0)$ of the dynamics. In fact 
\[
\mu = \det(I_d + D^2(F + \psi)(0)). 
\]
\end{remark}

\begin{proposition}
\label{prop:min_eigenvalue}
There exists $C>1$ such that for $x \in U$
\begin{equation}
	\label{eq:A-lower}
	\cA_{n, x}^* \geq C^{-1}I_{nd}.
\end{equation}
\end{proposition}

The main idea behind Proposition~\ref{prop:Hessian-bound} is that for each $x$, the sequence $x^*_{-k}(x)$ converges to $0$ exponentially, since they are the $x$ component of a backward orbit on the unstable manifold of $(0, 0)$. We can represent $\log \det \cA^*_{n,x}$ as a sum over the orbit (this connection has already appeared in \cite{Aubry1992}, see also \cite{Anantharaman2004}), which becomes a uniformly convergent sum. Proposition~\ref{prop:min_eigenvalue} also exploits this connection. 

We need a few lemmas for Proposition~\ref{prop:Hessian-bound}. 
\begin{lemma}
\label{lem:block}
Consider
\begin{equation}
	\label{eq:block_matrix}
	D_n=
	\begin{bmatrix}
	A_1 & -I_d &  & & & \\
	-I_d & A_2 & -I_d & & & \\
	& \ddots & \ddots & \ddots & \\
	&        & -I_d & A_{n-1} & -I_d \\
	&        &    &   -I_d & A_n \\ 
	\end{bmatrix}
\end{equation}
where $I_d, A_i$ are $d\times d$ matrices. Then
\begin{equation}
	\label{eq:det_dn}
	\det D_n=\det\left[\
	\begin{pmatrix}
	A_n & -I_d\\
	I_d & O_d
	\end{pmatrix}
	\begin{pmatrix}
	A_{n-1} & -I_d\\
	I_d & O_d
	\end{pmatrix}\cdots
	\begin{pmatrix}
	A_1 & -I_d\\
	I_d & O_d
	\end{pmatrix}
	\right]_{11}
\end{equation}
where $[\cdot]_{11}$ denote the top left element of the $2\times 2$ block matrix. Conjugate $
\begin{pmatrix}
A_i & -I\\
I & O
\end{pmatrix}
$ with $
\begin{pmatrix}
I & -I\\
O & -I
\end{pmatrix}
$, \eqref{eq:det_dn} becomes:
\begin{equation}
	\label{eq:matrix}
	\det D_n=\det\left[\
	\begin{pmatrix}
	A_n-I & I\\
	A_n-2I & I
	\end{pmatrix}
	\cdots
	\begin{pmatrix}
	A_1-I & I\\
	A_1-2I & I
	\end{pmatrix}
	\begin{pmatrix}
	I\\
	I
	\end{pmatrix}
	\right]_{1}
\end{equation}
\end{lemma}
\begin{proof}
We consider the equation
\begin{equation}
	\label{eq:21}
	(-\lambda I_{nd}+D_n)
	\begin{bmatrix}
	x_1\\
	\vdots\\
	x_n
	\end{bmatrix}=0
\end{equation}
Expand in components, we have
\begin{equation}
	\label{eq:22}
	\begin{pmatrix}
	x_i\\
	x_{i-1}
	\end{pmatrix}=
	\begin{pmatrix}
	-\lambda I_d+A_{i-1} & -I_d\\
	I_d & O_d
	\end{pmatrix}
	\cdots
	\begin{pmatrix}
	-\lambda I_d+A_1 & -I_d\\
	I_d & O_d
	\end{pmatrix}
	\begin{pmatrix}
	x_1\\
	0
	\end{pmatrix}
\end{equation}
for $2\leq x\leq n-1$, and
\begin{equation}
	\label{eq:22}
	\begin{pmatrix}
	0\\
	x_n
	\end{pmatrix}=
	\begin{pmatrix}
	-\lambda I_d+A_n & -I_d\\
	I_d & O_d
	\end{pmatrix}
	\cdots
	\begin{pmatrix}
	-\lambda I_d+A_1 & -I_d\\
	I_d & O_d
	\end{pmatrix}
	\begin{pmatrix}
	x_1\\
	0
	\end{pmatrix}
\end{equation}
Denote
\begin{equation}
	\label{eq:23}
	M(\lambda)=\begin{pmatrix}
	-\lambda I_d+A_n & -I_d\\
	I_d & O_d
	\end{pmatrix}
	\cdots
	\begin{pmatrix}
	-\lambda I_d+A_1 & -I_d\\
	I_d & O_d
	\end{pmatrix}, 
\end{equation}
then $(x_1, \ldots, x_n) \in \ker (-\lambda I_{nd}+D_n)$ if and only if $x_1 \in \ker [M(\lambda)]_{11}$. Let 
\[
	p^{A_1, \ldots, A_n}(\lambda) = \det (-\lambda I_{nd}+D_n), \quad 
	q^{A_1, \ldots, A_n}(\lambda) = \det[M(\lambda)]_{11}, 
\]
then these two polynomials have the same degree, roots, and leading coefficients. They must be equal if the roots are simple. We claim that the roots of $p^{A_1, \ldots, A_n}(\lambda)$ are simple on an open set of $A_1, \ldots, A_n$. Indeed, consider $A_1, \ldots, A_n \gg 1$ and such that $\diag\{A_1, \ldots, A_n\}$ has distinct eigenvalues whose mutual distance are also much larger than $1$, then $D_n$ has distinct eigenvalues robustly. Since $p^{A_1, \ldots, A_n}(0)$ and $q^{A_1, \ldots, A_n}(0)$ are polynomials of the coefficients of $A_1, \ldots, A_n$ and they agree on an open set, they must be equal to each other. 
\end{proof}

To study the Hessian along the minimizers, notice
\begin{equation}
	\label{eq:74}
	D\Phi(x,v)=
	\begin{pmatrix}
	I_d+D^2F(x) & I_d\\
	D^2F(x) & I_d
	\end{pmatrix}
\end{equation}
and 
\begin{equation}
	\label{eq:75}
	D\Phi(x_{-n},v_{-n})
	\begin{pmatrix}
	I_d\\
	D^2\psi(x_{-n})
	\end{pmatrix}=
	\begin{pmatrix}
	D^2F(x_{-n})+D^2\psi(x_{-n}) & I_d\\
	D^2F(x_{-n})+D^2\psi(x_{-n})-I_d & I_d
	\end{pmatrix}
	\begin{pmatrix}
	I_d\\
	I_d
	\end{pmatrix}
\end{equation}
Apply \eqref{eq:matrix} to \eqref{eq:hess}, we have
\begin{equation}
	\label{eq:76}
	\begin{split}
	\det \cA(x_{-n},\cdots,x_{-1})&=\det(D^2_{x_{-n},\cdots,x_{-1}}H(x_{-n},\cdots,x_{-1}))\\
	&=\det\pi_1D\Phi^{n}(x_{-n},v_{-n})
	\begin{pmatrix}
	I_d\\
	D^2\psi(x_{-n})
	\end{pmatrix}
	\end{split}
\end{equation}
Since $(x, \nabla \psi(x))$ is contained in the $\Phi$ invariant manifold $W^u$, and $(x_i^*, \nabla \psi(x_i^*))$ is an orbit of $\Phi$, the plane bundle $
\begin{pmatrix}
I_d\\
D^2\psi(x_i^*)
\end{pmatrix}
$ is invariant under $D\Phi$. Hence
\begin{equation}
	\label{eq:77}
	D\Phi(x^*_i,\nabla\psi(x^*_i))
	\begin{pmatrix}
	I_d\\
	D^2\psi(x^*_i)
	\end{pmatrix}
	=\left(I_d+D^2(F + \psi)(x^*_i)\right)
	\begin{pmatrix}
	I_d\\
	D^2\psi(x^*_{i})
	\end{pmatrix}
\end{equation}
Therefore,
\begin{equation}
	\label{eq:hess_prod}
	\det \cA_n^*(x)=\det\prod_{i=-n}^{-1}(I_d+D^2(F+\psi)(x^*_i))
\end{equation}

\begin{proof}[Proof of Proposition~\ref{prop:Hessian-bound}]
  Denote $\mu=\det\left(I_d+D^2(F+\psi)(0)\right)$, we have $\mu>1$. Since $x_{-n}^*$ converge to $0$ exponentially fast as $n\to\infty$, there exists $N>0$
  and constants $C_1>0$ such that for all $n \in \N$, we have
  \[
    \left|\det\left(I_d+D^2(F+\psi)(x_{-n}^*)\right)-\mu\right|<\frac{C_1}{n},
  \]
  which implies
  \[
(\mu-\frac{C_1}{n})^n<\det\cA_n^*(x)<(\mu+\frac{C_1}{n})^n.
\]

Therefore,
\[
  e^{-C_1\mu}\mu^n<\det\cA^*_n(x)<e^{C_1\mu}\mu^n
\]
\end{proof}

To prove Proposition \ref{prop:min_eigenvalue}, we need the following classical result, see \cite{Bellman97}, for example.
\begin{lemma}[Poincar\'{e} Seperation Theorem]
\label{lem:seperation}
Let $A$ be an $n\times n$ symmetric matrix, $(u_1,\cdots,u_r)$ be an orthonormal set in $\R^n$, $r<n$. Define $B=(u_i^TAu_j)$. Let
\begin{equation}
	\label{eq:31}
	\lambda_1(A)\leq\lambda_2(A)\leq\cdots\leq\lambda_n(A)
\end{equation}
be the eigenvalues of $A$ and let
\begin{equation}
	\label{eq:32}
	\lambda_1(B)\leq\cdots\leq\lambda_r(B)
\end{equation}
be the eigenvalues of $B$. Then
\begin{equation}
	\label{eq:33}
	\lambda_k(A)\leq\lambda_k(B)\leq\lambda_{k+n-r}(A),\quad 1\leq k\leq r.
\end{equation}
\end{lemma}

\begin{proof}[Proof of Proposition~\ref{prop:min_eigenvalue}]

There exists $\delta>0$ such that $D^2F(0)\geq\delta$. Since $F\in C^3$ and $\lim_{n\to\infty}x^*_{-n}\to 0$, there exists some $N$ such that
$D^2F(x^*_{-n})\geq\frac{\delta}{2}$ for $n>N$.

Denote
\begin{equation}
	\cA_n=\cA_{n}^{(j)}=
	\begin{bmatrix}
	A_{n} & -I_d &  & & & \\
	-I_d & A_{n-1} & -I_d & & & \\
	& \ddots & \ddots & \ddots & \\
	&        & -I_d & A_{2} & -I_d \\
	&        &    &   -I_d & A_1 \\ 
	\end{bmatrix}
\end{equation}  
and
\begin{equation}
	B_{N,n}=
	\begin{bmatrix}
	A_{n} & -I_d &  & & & \\
	-I_d & A_{n-1} & -I_d & & & \\
	& \ddots & \ddots & \ddots & \\
	&        & -I_d & A_{N+2} & -I_d \\
	&        &    &   -I_d & A_{N+1} \\ 
	\end{bmatrix}
\end{equation}
Then we have
\begin{equation}
	\label{eq:38}
	B_{N,n}\geq     \begin{bmatrix}
	(2+\frac{\delta}{2})I_d & -I_d &  & & & \\
	-I_d & (2+\frac{\delta}{2})I_d & -I_d & & & \\
	& \ddots & \ddots & \ddots & \\
	&        & -I_d & (2+\frac{\delta}{2})I_d & -I_d \\
	&        &    &   -I_d & (2+\frac{\delta}{2})I_d \\ 
	\end{bmatrix}
\end{equation}
As a result, the minimum eigenvalue of $B_{N,n}$ is bounded below by a constant independent of $n$.

In the follwing, we use $C(N,d)$  to denote any constant that depends on $N$ and $d$ but does not depend on $n$.

By \eqref{eq:hess_prod}, we have
\begin{equation}
	\label{eq:34}
	\frac{\det\cA_n}{\det B_{N,n}}=\det\prod_{i=-N}^{-1}(I_d+D^2F(x_i^{(j)})+D^2\psi(x_i^{(j)})),
\end{equation}
which is bounded below and above by a constant independent of $n$.

Let $\lambda_1\leq\cdots\leq\lambda_{nd}$ be the eigenvalues of $\cA_{n}$ and let $\mu_1\leq\cdots\leq\mu_{(n-N)d}$ be the eigenvalues of $B_{N,n}$. 
By \eqref{eq:34}, we have
\begin{equation}
	\label{eq:36}
	\frac{\prod_{i=1}^{nd}\lambda_i}{\prod_{i=1}^{(n-N)d}\mu_i}=\frac{\det\cA_n}{\det B_{N,n}}\geq C(N,d).
\end{equation}
By Lemma \ref{lem:seperation}, we have
\begin{equation}
	\label{eq:35}
	\lambda_k\leq\mu_k\leq\lambda_{k+Nd}.
\end{equation}
Therefore
\begin{equation}
	\label{eq:37}
	\lambda_1\geq C(N,d)\frac{\prod_{i=q}^{(n-N)d}\mu_i}{\prod_{i=2}^{nd}\lambda_j}\geq C(N,d)\frac{\mu_1}{\prod_{i=(n-N)d+1}^{nd}\lambda_j}
\end{equation}
Since $\mu_1$ is bounded from below and $\lambda_i$'s are bounded from above by constants independent of $n$, we have
\begin{equation}
	\label{eq:39}
	\lambda_1\geq C(N,d). \qedhere
\end{equation}
\end{proof}

\section{Laplace's method for the partition function}
\label{sec:laplace}

In this section, we prove Proposition~\ref{prop:laplace}, which establishes the estimate 
\[
C^{-1} \le \frac{\tL_\nu^n \bOne(x)}{Q_n} \le C \chi_\nu(x)
\]
for $0 \le n \le N_1(\nu)$, where $\chi_\nu(x) = 1$ for $x \in U$ and $\nu^{-d/2}$ for $x \notin U$. The plan of this section is as follows: 
\begin{itemize}
 \item We first prove some technical lemmas on the function $H_{n, x}$ (Lemma~\ref{lem:H-integrable} to \ref{lem:perturbed_hessian});  
 \item We then prove Proposition~\ref{prop:U-laplace} which establishes the estimate for $x \in U$; 
 \item After that we give the proof of Proposition~\ref{prop:laplace}.  
\end{itemize}

\begin{lemma}\label{lem:H-integrable}
There is a constant $C > 0$ depending only on $F$  such that
\[
	\int e^{-H_{n,x}(X)} dx_{-n} \cdots dx_{-1} \le C^{nd}. 
\]
\end{lemma}
\begin{proof}
Let $C_0 = \|F\|_{C^0} + \|\psi\|_{C^0}$, then 
\[
\begin{aligned}
& \int e^{-H_{n, x}(X)} dX \\
& = \int \exp \left( -\sum_{k = -n}^{-1} (\frac12 |x_{k + 1} - x_k|^2 + F(x_k)) - \psi(x_{-n} + \psi(x_0))\right) dx_{-n} \cdots dx_{-1}  \\
& \le e^{n C_0} \int \exp \left( - \frac12 \sum_{k = -n}^{-1} \frac12 |x_{k + 1} - x_k|^2 )\right) dx_{-n} \cdots dx_{-1} = e^{nC_0} (2\pi)^{nd/2}. 
\end{aligned}
\]
The lemma follows by setting $C = e^{C_0}(2\pi)^{d/2}$. 
\end{proof}

\begin{lemma}\label{lem:int-shift-lower}
There exists $\delta > 0$ depending only on $F$ such that for all nonzero $l \in \Z^{nd}$, $n\in\N$ and $x \in U$, 
\[
	H_{n, x}(X^*(x) + l)  - H_{n, x}(X^*(x)) = H_{n, x}(X^*(x) + l) > \delta. 
\]
\end{lemma}
\begin{proof}
By the definition of $U$ (see Proposition~\ref{prop:U-non-deg}), for every $x \in \overline{U}$, the minimum $h(\cdot, x)$ is achieved at a unique point $\bar{y}(x)$. As a result, there exists $\delta > 0$ such that 
\[
h(\bar{y}(x) + q, x) > \delta, \quad \text{ for all nonzero } q \in \Z^d. 
\]

Let $j = \max\{k \st l_k \ne 0, \, -n \le k \le -1\}$, then 
\[
\begin{aligned}
H_{n, x}(X^*(x) + l)
& = \sum_{k = -n}^{-1} \tilde{h}(x_k^* + l_k, x_{k+1}^* + l_{k + 1})
\ge \sum_{k = j}^{-1} \tilde{h}(x_k^* + l_k, x_{k+1}^* + l_{k + 1}) \\
& = \tlh(x_j^* + l_j, x_{j+1}^*) > \delta
\end{aligned}
\]
where we used $\tlh \ge 0$ and $x_j^* = \bar{y}(x_{j+1}^*)$. 
\end{proof}

\begin{lemma}\label{lem:H-inf-lower}
For any $r>0$, there exists $\delta > 0$ depending only on $r$ and $F$ such that
\[
	\inf\{H_{n, x}(X) \st x \in U, \, \|X - X^*(x)\|_\infty > r, \, n \in \N \} > \delta > 0. 
\]
\end{lemma}
\begin{proof}
Write $\|X\|_{\T, \infty} = \sup_{k = 1}^n |x_{-k}|_\T$. By Lemma~\ref{lem:int-shift-lower}, it suffices to prove 
\[
	\inf\{H_{n, x}(X) \st x \in U, \, \|X - X^*(x)\|_{\T, \infty} > r, \, n \in \N \} > \delta. 
\]
First of all, there exists $\delta_1 > 0$ such that  $\inf_{|x|_\T > r}\psi(x) \ge 2\delta_1$ for all $|x| > r$. Secondly, by Proposition~\ref{prop:weak-KAM},  $\|T^n \bOne - \psi\|_* \to 0$ as $n \to \infty$, therefore there exists $N \in \N$ such that for all $n \ge N$, 
\[
	\|T^n \bOne - \psi\|_* < \delta_1. 
\]
Thirdly, since $H_{N, x}(X)$ has a unique minimum at $X^*$,  by a compactness argument, for a fixed $N$, there exists $\delta_2> 0$ depending on $N$ such that 
\[
	\inf \{  H_{N, x}(X) \st  x \in U, \,  \|X - X^*\|_{\T, \infty} > r, \, X \in (\R^d)^N\} > \delta_2. 
\]

Suppose $\|X - X^*\|_{\T, \infty} > r$, then there exists $1 \le m \le n$ such that $|x_{-m} - x_{-m}^*|_\T > r$. If $m < N$, we write $X_{-N}^{-1} = (x_k)_{k = -N}^{-1}$, then $\|X_{-N}^{-1} - X^*\|_{\T, \infty} > r$, hence
\[
	H_{n, x}(X) = \sum_{k = -N}^{-1} \tlh(x_k, x_{k+1}) + \sum_{k = -n}^{-N-1} \tlh(x_k, x_{k+1}) \ge H_{N, x}(X_{-N}^{-1}) > \delta_2
\]
since $\tlh \ge 0$. If $m \ge N$, then 
\[
	\begin{aligned}
	H_{n, x}(X) & = -\psi(x) + \sum_{k = -m}^{-1} h(x_k, x_{k+1}) + \sum_{k = -n}^{-m-1} h_{x_k, x_{k+1}} + \psi(x_{-n}) \\
	& \ge - \|T^m \bOne - \psi\|_* + \psi(x_{-m}) > 2\delta_1 - \delta_1 > \delta_1.
	\end{aligned} 
\] 
Take $\delta = \min\{\delta_1, \delta_2\}$ and the lemma follows. 
\end{proof}

\begin{lemma}
\label{lem:perturbed_hessian}
Assume $\cA_n$ is a symmetric and positive definite $nd \times nd$ matrix,  and assume there is a uniform lower bound $\lambda_{min}>0$ of the smallest eigenvalue. Suppose $C \mu_{min} > 1$ and $0 < \epsilon < 1/(Cn)$, then
\begin{equation}
	\label{eq:27}
	\frac{\det (\cA_n-\epsilon I_{nd})}{\det \cA_n}\geq (1-\frac{1}{C\mu_{min}})^d\ \mathrm{and}\ \frac{\det (\cA_n+\epsilon I_{nd})}{\det \cA_n}\leq e^{\frac{d}{C\mu_{min}}}. 
\end{equation}
\end{lemma}
\begin{proof}
For each $\cA_n$, there exists invertible matrix $E_n$ such that
\begin{equation}
	\label{eq:28}
	B_n=E^{-1}_n\cA_nE_n
\end{equation}
is diagonal.  Denote $B_n=\mathrm{diag}(b_{n,j})_{1\leq j\leq nd}$. Then for $\epsilon<\frac{1}{Cn}$, we have
\begin{equation}
	\label{eq:29}
\begin{aligned}
	\frac{\det(\cA_n-\epsilon I_{nd})}{\det \cA_n} 
	&= \frac{\det(B_n-\epsilon I_{nd})}{\det B_n}
	=\frac{\prod_{j=1}^{nd}(b_{n,j}-\epsilon)}{\prod_{j=1}^{nd}b_{n,j}} \\
	& \geq (1-\frac{1}{C\mu_{min}n})^{dn}
	 \geq (1-\frac{1}{C\mu_{min}})^d.
\end{aligned}
\end{equation}
The other inequality is similar.
\end{proof}

\begin{proposition}\label{prop:U-laplace}
Let $\mu$ be as in Proposition~\ref{prop:Hessian-bound}. 
There exist constants $\nu_0 > 0$ and $C > 1$ such that for all $n \le N_1(\nu) = C^{-1} (\nu \log \frac{1}{\nu})^{-\frac13}$, 
\[
	C^{-1} < \lambda^{-n} \tL_\nu^n \bOne(x)  \le C, \quad \text{ for all } x \in U.
\]
\end{proposition}

\begin{proof}
Throughout the proof, the notation $C_k$ denote a constant that is greater than $1$ and depends only on $F$.

First we note that $H_{n, x}(x_{-n}, \cdots, x_{-1})$ is $C^3$ if all $x_k \in U$. Using Taylor expansion, we have
\[
	\begin{aligned}
	H_{n, x}(X) & = 
	H_{n, x}(X^*) + \frac12(X - X^*)^T A_{n, x}(X^*) (X - X^*)  \\
	& \quad + \sum_{|\beta| = 3} \left( \sum_{k = -n}^{-1} D^3 F(\xi_k)(x_k - x_k^*)
	- D^3 \psi(\xi_k)(x_k - x_k^*) \right), 
	\end{aligned}
\]
where $\xi_k$ are intermediate points, and $D^3 F$, $D^3 \psi$ are trilinear forms.  When $\|X - X^*\|_\infty < r$ we have 
\begin{equation}
	\label{eq:8}
	\sum_{i={-n}}^{-1}|x_i-x_i^*|^3\leq r\sum_{i=-n}^{-1}|x_i-x_i^*|^2
\end{equation}
and 
\begin{equation}
	\label{eq:9}
	|H_{n,x}(X) - H_{n,x}(X^*) - \frac{1}{2}(X-X^*)^T \cA_{n, x}^*(X-X^*)| \leq C_0r \|X-X^*\|_2^2
\end{equation}
where $C_0 = \frac16 \max_{y\in \T^d}(\|DF^3(y)\| + \|D^3\psi(y)\|)$.

By Proposition~\ref{prop:min_eigenvalue}, there exists $r_0>0$ and $C_1 > 1$ depending only on $F$, such that
\begin{equation}
  \label{eq:hessian_r0}
	H_{n,x}(X) > C_1^{-1} \|X - X^*\|_2^2, \quad \text{ if } \|X - X^*\|_\infty < r_0. 
\end{equation}
By Lemma~\ref{lem:H-inf-lower}, there exists $\delta > 0$ such that 
\begin{equation}
  \label{eq:Hnx-delta}
  	H_{n, x}(X) \ge \delta, \quad \text{ if } \|X - X^*\|_\infty \ge r_0. 
\end{equation}

Let $r_1(n) = \frac{1}{C_2 n}$, where $C_2$ is large enough so that Lemma~\ref{lem:perturbed_hessian} applies with $\cA = \cA_{n, x}^*$, $C = C_1$ and $\epsilon = C_1 r$ (note that the minimum eigenvalue of $\cA_{n, x}^*$ is at least $C_1^{-1}$ due to \eqref{eq:hessian_r0}). Set 
\[
	S_0 = \left\{  X \in (\R^d)^n \st \|X - X^*\|_\infty \le  r_0 \right\}, 
\]
\[
	S_1= S_1(n) = \left\{  X \in (\R^d)^n \st \|X - X^*\|_\infty \le  r_1(n) \right\}. 
\]
Denote $\tlX=X-X^*$, we have the following estimates,
\begin{equation}
	\label{eq:int-S-upper}	\begin{split}
	&\int_{S_1}\exp\left[ -\frac{1}{2\nu}H_{n,x}(X)\right]dx_{-n}\cdots dx_{-1}\\
	& \leq \int_{||\tlX||_{\infty}<r_1}\exp\left[-\frac{1}{2\nu}\left(H_{n,x}(X^*)+\frac{1}{2}(\tlX)^T(\cA_{n, x}^*- C_0 r_1 I_{nd})\tlX\right)\right] d\tilde{X}\\
	&\leq (4\pi\nu)^{nd/2}\det(\cA_{n, x}^*- C_0 r_1 I_{nd})^{-1/2}\\
	&\leq (4\pi\nu)^{nd/2}\det(\cA_{n, x}^*)^{-1/2}\left(\frac{\det(\cA_{n, x}^*- C_0 r_1 I_{nd})}{\det(\cA_{n, x}^*)}\right)^{-1/2}\\
	&\leq C_3 (4\pi\nu)^{nd/2}\det(\cA_{n, x}^*)^{-1/2} \leq C_4  (4\pi\nu)^{nd/2} \mu^n
	\end{split}
\end{equation}
for some $C_3, C_4 > 1$. In the last line we applied Lemma \ref{lem:perturbed_hessian} and Proposition~\ref{prop:Hessian-bound}. 

We now estimate the same integral from below. Indeed, we have
\begin{equation}
  \label{eq:S1-lower}
  \begin{aligned}
 & \int_{S_1} \exp \left( - \frac{1}{2\nu} H_{n, x}(X) \right) dX \\
& \ge \int_{\|\tilde{X}\|_\infty \le r_1(n)}
\exp\left( - \frac{1}{4\nu} \tilde{X}^T (\cA_{n, x}^* - C_0 r_1 I_{nd}) \tilde{X} \right) d \tilde{X} \\
& = (2\nu)^{nd/2} \int_{\|V\|_\infty \le r_1(n)/\sqrt{2\nu}}
	\exp\left(  - \frac12 V^T (\cA_{n, x}^* - C_0 r_1 I_{nd})V \right) dV \\
& \ge (4\pi \nu)^{nd/2} \det(\cA_{n, x}^* - C_0 r_1 I_{nd})^{-\frac12}
\cdot 
(2\pi)^{-nd/2} \int_{\|V\|_\infty \le \frac{\|r_1(n)\|}{C_1 \sqrt{2\nu}}} \exp(- \frac12 V^T V) dV \\
& \ge 
 (4\pi \nu)^{nd/2} \det(\cA_{n, x}^* - C_0 r_1 I_{nd})^{-\frac12}
 \left( 1 - \frac{C_1 \sqrt{2\nu}}{r_1(n)} e^{-\frac12 \frac{r_1^2(n)}{2C_1^2 \nu}}\right)^{nd} \\
& \ge 
C_5^{-1} (4\pi \nu)^{nd/2} \mu^n
\left( 1 - C_1^2 n\sqrt{2\nu} e^{-\frac12 (C_1^2 n \sqrt{2\nu})^2} \right). 
\end{aligned}
\end{equation}
In the last formula, we applied the Gaussian tail bound $\sqrt{1}{2\pi} \int_{|x| > r} e^{-\frac12 x^2} < \frac{1}{r\sqrt{2\pi}} e^{-\frac12 r^2}$. 

Suppose $n < \nu^{-\frac13}$, we will choose $\nu_0$ small enough depending only on $F$ such that 
\[
C_1^2 n\sqrt{2\nu} e^{-\frac12 (C_1^2 n \sqrt{2\nu})^2} < \frac{1}{2n}. 
\]
Indeed, 
\[
2C_1^2 n^2 \sqrt{2\nu} e^{-\frac12(C_1^2 n \sqrt{2\nu})^2}
< 2\sqrt{2} C_1^2 \nu^{-\frac16} e^{-\frac12(C_1^2 \sqrt{2} \nu^{-\frac16})^2} < 1
\]
if $\nu_0$ is small enough. It follows that 
\begin{equation}
	\label{eq:int-S-lower}
	\begin{split}
	\int_{S_1}\exp\left[ -\frac{1}{2\nu}H_{n,x}(X)\right]dx_{-n}\cdots dx_{-1}& \geq 
	C_5^{-1} (4\pi \nu)^{nd/2} \mu^n (1 - \frac{1}{2n})^{nd}
	\end{split}
\end{equation}
Summarizing, there exist $C_6 > 1$ such that 
\begin{equation}
	\label{eq:18}
	C_6^{-1} \mu^n
	\leq  \frac{1}{(4\pi\nu)^{nd/2}}  \int_{S_1} \exp\left[-\frac{1}{2\nu} H_{n,x}(X) \right] dx_{-n}\cdots dx_{-1}
	\leq C_6 \mu^n. 
\end{equation}

Moreover, the integral over $S_0^c$ can be estimated as follows:
\begin{equation}
	\label{eq:S0-complement}
	\begin{split}
	& \int_{S^c_0}\exp\left[-\frac{1}{2\nu}H_{n,x}(X)\right]dx_{-n}\cdots dx_{-1} \\
	& \leq \exp\left[(1-\frac{1}{2\nu}) \delta\right]\int_{\R^n}e^{-H_{n,x}(X)}dx_{-n}\cdots dx_{-1}
	\leq \exp\left[(1-\frac{1}{2\nu}) \delta\right]C^{nd}
	\end{split}
\end{equation}
where the last estimate is due to Lemma~\ref{lem:H-integrable}. 

We prove our proposition by splitting into two cases. 

\textbf{Case 1}: $r_1(n) \ge r$. In this case $S_0^c \supset S_1^c$, we bound the integral on $\R^{nd}$ by the integral on $S_0^c$ and $S_1$ to get 
\[
	\begin{aligned}
	C_6^{-1} \mu^n
	& \leq  \frac{1}{(4\pi\nu)^{nd/2}}  \int \exp\left[-\frac{1}{2\nu} H_{n,x}(X) \right] dx_{-n}\cdots dx_{-1} \\
	& \leq  C_6 \mu^n + (4\pi \nu)^{-nd/2} \exp\left[(1-\frac{1}{2\nu}) \delta\right]C^{nd}. 
	\end{aligned}
\]
Our conclusion holds if we can show
\[
	\frac{1}{(4\pi\nu)^{nd/2}}e^{(1-\frac{1}{2\nu}) \delta}C^{nd}\leq C_6 \mu^n. 
\]
We assume that $\nu_0$ is small enough such that $\nu_0^{nd/2} (4\pi)^{-nd/2} e^\delta  C^{nd} \mu^{-n} C_6^{-1} \le 1$, then it suffice to prove 
\[
	\nu^{-nd} e^{-\delta/(2\nu)} \le 1, \quad \text{ or } \quad
	nd \log \nu^{-1} \le \frac{\delta}{2\nu}, 
\]
which holds if $n \le \frac{\delta}{2d} (\nu \log \frac{1}{\nu})^{-1}$. 

\textbf{Case 2}: $r_1(n) < r$. In this case, we split the integral into three domains:
\[
\int_{S_1} + \int_{S_0 \setminus S_1} + \int_{S_0^c}. 
\]
First we assume the same conditions on $\nu_0$ and $n$ so that 
\[
	(4\pi \nu)^{-nd/4}\int_{S_0^c}\exp\left[-\frac{1}{2\nu}H_{n,x}(X)\right]dx_{-n}\cdots dx_{-1} < C_7 \lambda^n 
\]
for some $C_7 > 1$. For the integral on $S_0 \setminus S_1$, we use a similar computation to \eqref{eq:S0-complement} to get 
\[
	\int_{S_0 \setminus S_1}\exp\left[-\frac{1}{2\nu}H_{n,x}(X)\right]dx_{-n}\cdots dx_{-1}
	\leq \exp\left[(1-\frac{1}{2\nu}) C_0^{-1} r_1^2\right]C^{nd}. 
\]
 We will show that under our conditions
\[
	(4\pi \nu)^{-nd/2} \exp\left[(1-\frac{1}{2\nu}) C_0^{-1} r_1^2\right]C^{nd} \le C_6 \lambda^n. 
\]
Indeed, plug in $r_1(n) = 1/(C_2 n)$, and similarly as in Case 2, we set $\nu_0$ small enough so that all the $\nu$-independent exponential terms are dominated by $\nu^{-nd/2}$, then it suffices to prove 
\[
	\nu^{-nd} e^{-C_0^{-3} n^{-2} (2\nu)^{-1}} \le 1, \quad 
	\text{ or } \quad
	(nd)\log \frac{1}{\nu} \le C_0^{-3} n^{-2} \nu^{-1}, 
\]
which holds if $n \le (C_0^3d)^{-\frac13} (\nu \log \frac{1}{\nu})^{-\frac13}$. 

Finally, we note that all our estimates hold if $n \le C^{-1} (\nu \log \frac{1}{\nu})^{-\frac13}$ with $C = \max\{(C_0^3d)^{\frac13}, \frac{\delta}{2d}\}$, and $\nu_0$ sufficiently small depending only on $F$. 
\end{proof}

\begin{proof}[Proof of Proposition~\ref{prop:laplace}]
In this proof, all constants $C_k$ depend only on $F$. 

It suffices to prove the estimate for $x \notin U$. Let $C$ be the constant in Proposition~\ref{prop:U-laplace}. For $n \le N_1(\nu) = C^{-1}(\nu \log\frac{1}{\nu})^{-\frac13}$, set
\[
	M_\nu^n = \mu^{-n} \nu^{\frac{d}{2}} \sup_{x \notin U} \tL_\nu^n \bOne(x). 
\]
First of all, if $x \notin U$ and $\nu < 1$, 
\[
	\tL_\nu \bOne(x) = (4\pi \nu)^{-\frac{d}{2}} \int e^{-\frac{1}{2\nu} \tlh(y, x)} dy 
	\le (4\pi \nu)^{-\frac{d}{2}} \int e^{-\frac{1}{2} \tlh(y, x)} dy \le C_1 \nu^{-\frac{d}{2}}
\]
for some $C_1 > 1$ by Lemma~\ref{lem:H-integrable}(for $n=1$). We conclude that 
\[
	M_\nu^1 \le C_1 \mu^{-1}. 
\]

On the other hand, (recall $\tlK_\nu(y, x) = (4\pi \nu)^{-\frac{d}{2}} e^{-\tlh(y, x)/(2\nu)}$), 
\[
	\tL_\nu^n \bOne(x) = \int \tlK_\nu(y, x) \tL_\nu^{n-1} \bOne(y) dy = \int_U + \int_{U^c}. 
\]
By Proposition~\ref{prop:U-laplace}, 
\[
	\mu^{-n} \int_U \tlK_\nu(y, x) \tL_\nu^{n-1} \bOne(y) dy \le C \mu^{-1} \int_U \tlK_\nu(y, x) dy 
	\le C \mu^{-1} \tL_\nu \bOne(x) \le C C_1 \mu^{-1} \nu^{- \frac{d}{2}}. 
\]
By Proposition~\ref{prop:U-non-deg}, 
\[
	\begin{aligned}
	& \mu^{-n} \int_{U^c} \tlK_\nu(y, x) \tL_\nu^{n-1} \bOne(y) dy
	\le \mu^{-n} (4\pi \nu)^{-\frac{d}{2}} \int_{U^c} e^{-\frac{\tlh(y, x)}{2\nu}} \tL_\nu^{n-1} \bOne(y) dy \\
	& \le \mu^{-1} (4\pi)^{-\frac{d}{2}}\nu^{-d} \int_{U^c} e^{- \delta(\frac{1}{2\nu} - 1)} e^{-\tlh(y, x)} M_\nu^{n-1} dy 
	\le C_2 \mu^{-1} \nu^{-d}e^{-\frac{\delta}{2\nu}} M_\nu^{n-1} 
	\end{aligned}
\]
for some $C_2 > 1$. 
Suppose $\nu_0$ is small enough such that $C_2 \mu^{-1} \nu^{-d}e^{-\frac{\delta}{2\nu}} < \frac12$, then 
\[
	M_\nu^n \le C C_1 \mu^{-1}  + \frac12 M_\nu^{n-1}, 
\]
which implies $M_\nu^n \le 2CC_1 \mu^{-1}$ for all $n \le N(\nu)$. 

For the lower bound, let $x \in \R^d$ and assume $\bary \in \argmin \tlh(\cdot, x)$. Using the semi-concavity of $\psi$, there is $C_3 > 0$ such that $\tlh(y, x) \le C_3|y - \bary|^2$ for all $y \in \R^d$. Then
\[
	\begin{aligned}
	\mu^{-n} \tL_\nu^n \bOne(x) 
	& \ge \mu^{-n }\int_U \tlK_\nu^n(y, x) \tL_\nu^{n-1} \bOne(y) dy \ge C^{-1} \mu^{-1}\int_U \tlK_\nu^n(y, x) dy \\
	& \ge C^{-1} \mu^{-1} (4\pi \nu)^{-\frac{d}{2}} \int_{y \st |y - \bary(x)| \le \nu^{\frac12}} e^{- C_3|y - \bary(x)|^2/(2\nu)} dy \\
	& = C^{-1} \mu^{-1} (2\pi)^{-\frac{d}{2}} \int_{|v| \le 1} e^{-C_3|v|^2} dy = C^{-1} \mu^{-1} C_4
	\end{aligned}
\]
for some $C_4 > 1$. The Proposition follows.
\end{proof}

\bibliographystyle{abbrv}
\bibliography{exp-convergence}

\end{document}